\documentclass[final,leqno]{siamltex}
\usepackage{amsmath,amssymb}
\usepackage{graphicx,color}
\usepackage{enumerate}
\usepackage{verbatim}

\newtheorem{remark}[theorem]{Remark}

\newcommand{\ta}{\tilde{\alpha}}

\newcommand{\om}{\omega}
\newcommand{\Om}{\Omega}
\newcommand{\pa}{\partial}
\renewcommand{\i}{{\rm\mathbf i}}
\newcommand{\bfi}{{\rm\mathbf i}}
\DeclareMathOperator{\re}{{Re}}
\DeclareMathOperator{\im}{{Im}}
\newcommand{\ddiv}{\mbox{\rm div\,}}

\newcommand{\bV}{\mathbf{V}}
\newcommand{\bH}{\mathbf{H}}
\newcommand{\bff}{\mathbf{f}}

\newcommand{\bfn}{\mathbf{n}}

\newcommand{\bftu}{\tilde{\mathbf{u}}}

\newcommand{\bfu}{\mathbf{u}}
\newcommand{\bfv}{\mathbf{v}}
\newcommand{\bfw}{\mathbf{w}}
\newcommand{\bfx}{\mathbf{x}}

\newcommand{\bfF}{\mathbf{F}}

\newcommand{\bfS}{\mathbf{S}}

\newcommand{\bfU}{\mathbf{U}}
\newcommand{\bfnu}{\boldsymbol{\nu}}
\newcommand{\bfpsi}{\boldsymbol{\psi}}
\newcommand{\bfPsi}{\boldsymbol{\Psi}}
\newcommand{\bftPsi}{\widetilde{\bfPsi}}
\newcommand{\bfphi}{\boldsymbol{\phi}}
\newcommand{\bfPhi}{\boldsymbol{\Phi}}

\newcommand{\bnabla}{\boldsymbol{\nabla}}
\newcommand{\bsnabla}{\boldsymbol{\nabla}_{s}}
\newcommand{\bdiv}{\mbox{\rm \bf div\,}}
\newcommand{\ome}{\omega}
\newcommand{\Ome}{\Omega}
\newcommand{\bL}{\mathbf{L}}
\newcommand{\Langle}{\bigl\langle}
\newcommand{\Rangle}{\bigr\rangle}

\newcommand{\norm}[1]{\left\Vert#1\right\Vert}

\newcommand{\E}{\mathbb{E}}
\newcommand{\cE}{\mathcal{E}}
\newcommand{\cS}{\mathcal{S}}
\newcommand{\cT}{\mathcal{T}}

\newcommand{\veps}{\varepsilon}
\newcommand{\p}{\partial}

\renewcommand{\i}{{\rm\mathbf i}}
\renewcommand{\Im}{{\rm Im }}
\renewcommand{\Re}{{\rm Re }}

\newcommand{\cF}{\mathcal{F}}

\def\be{\begin{equation}}
\def\ee{\end{equation}}
\def\br{\begin{eqnarray}}
\def\er{\end{eqnarray}}

\def\Langle{\left\langle}

\def\Rangle{\right\rangle}

\title{An efficient Monte Carlo interior penalty discontinuous Galerkin method for elastic 
wave scattering in random media}
\author{
Xiaobing Feng\thanks{Department of Mathematics, The University of
Tennessee, Knoxville, TN 37996, U.S.A.  ({\tt xfeng@math.utk.edu}). The work of this author 
was partially supported by the NSF grant DMS-1318486.}
\and
Cody Lorton\thanks{Department of Mathematics and Statistics, University of
West Florida, Pensacola, FL 32514, U.S.A.  ({\tt clorton@uwf.edu}).}
}

\begin{document}

\maketitle

\begin{abstract}
This paper develops and analyzes an efficient Monte Carlo interior penalty discontinuous Galerkin
(MCIP-DG) method for elastic wave scattering in random media. The method is constructed based on
a multi-modes expansion of the solution of the governing random partial differential equations. It is 
proved that the mode functions satisfy a three-term recurrence system of partial differential equations
(PDEs) which are nearly deterministic in the sense that the randomness only appears 
in the right-hand side
source terms, not in the coefficients of the PDEs. Moreover, the same differential operator applies to
all mode functions. A proven unconditionally stable and optimally convergent IP-DG method is used 
to discretize the deterministic PDE operator, an efficient numerical algorithm is proposed based on  
combining the Monte Carlo method and the IP-DG method with the $LU$ direct linear solver. 
It is shown that the algorithm converges optimally with respect to both the mesh size $h$ and the 
sampling number $M$, and practically its total computational complexity only amounts to solving 
a few deterministic elastic Helmholtz equations using a Guassian elimination direct linear solver.  Numerical
experiments are also presented to demonstrate the performance and key features
of the proposed MCIP-DG method.
\end{abstract}



Elastic wave scattering problems arise from applications in a 
variety of fields including geoscience, image science, the petroleum industry, 
and the defense industry, to name a few.
Such problems have been 
extensively studied both analytically and numerically in the past several decades 
(cf. \cite{Fouque_Garnier_Papanicolaou_Solna_07, Ishimaru_97} and the references therein). 
The material properties of the elastic media in which the wave 
propagates play a principle role in the methods used to solve the elastic wave 
scattering problem.
Common medium characterizations include homogeneous and isotropic media, 
inhomogeneous and anisotropic media, and random media.
As the characterization of the media becomes more complicated,  so do the 
computations of the solutions of the associated wave equations.
In the case of random media, wave forms may vary significantly for different samplings and
as a result, stochastic quantities of interests such as the mean, variance, and/or higher order moments must often be sought. 

In this paper we are concerned with developing efficient numerical methods for 
solving the elastic Helmholtz equations with random coefficients, which models 
the propagation in random media of elastic waves with a fixed frequency. Specifically, 
we consider the following random elastic Helmholtz problem: 
\begin{align}
	-\ddiv \big(\sigma(\bfu(\om,\cdot)) \big) - k^2 \alpha^2(\om,\cdot) \bfu(\om,
	\cdot) & = \bff(\om,\cdot) && \mbox{in } D, \label{Eq:ElasticPDE}\\
	\sigma \big(\bfu(\om,\cdot)\big) \bfnu + \i k A \bfu(\om,\cdot) & 
	= \bf0 && \mbox{on } \pa D, \label{Eq:ElasticBC}
\end{align}
for a.e. $\omega\in \Omega$.  Here $\sigma$ is the stress tensor defined by
\begin{align*}
	\sigma\big(\bfu(\om,\cdot)\big) &:= 2 \mu \bsnabla\bfu(\om,\cdot) 
	+ \lambda \ddiv \bfu(\om,\cdot) I, \\
	\bsnabla\bfu(\om,\cdot) &:= \frac{1}{2} \big( \bnabla \bfu(\om,
	\cdot) + \bnabla \bfu(\om,\cdot)^T \big),
\end{align*}
and $A$ is a $d \times d$ constant SPD matrix.  $k>0$ denotes the frequency 
of the wave. $\i=\sqrt{-1}$ denotes the imaginary unit. $D\subset \mathbb{R}^d$ ($d = 2,3$) 
is a bounded domain with boundary $\pa D$, $\bfnu$ denotes the outward normal to $\pa D$. 
For each $\bfx \in D$, $\alpha(\cdot,\bfx)=\sqrt{\rho(\cdot,\bfx)}$ 
is a real-valued random variable defined over a probability 
space $(\Omega, \mathcal{F}, P)$,  where $\rho(\cdot,\bfx)\geq 0$ denotes the density 
of the random media which is the main source of randomness in the above PDEs. 
Thus, $k\alpha(\om,\cdot)$ characterizes a random wave number for the elastic medium $D$. 
We also note that the notation $\bsnabla\bfu(\om,\cdot)$ is often called
the strain tensor and is denoted by $\varepsilon(\bfu(\om,\cdot))$ in the literature.

In this paper we mostly focus on the case of weakly random media in the sense 
that the elastic medium is a small random perturbation of a 
homogeneous background medium, that is, $\alpha(\om, \cdot):= 1 + \veps \eta(\om,\cdot)$. 
Here $\veps > 0$ represents the magnitude of the random fluctuation and $\eta 
\in L^2(\Om, L^\infty(D))$ is some random field which has a compact support on $D$ 
and satisfies
\begin{align*}
	P\Big \{ \om \in \Om \,\,;  \| \eta(\om, \cdot) \|_{L^\infty(D)} 
	\leq 1 \Big\} = 1.
\end{align*}
At the end of the paper, we will also present an idea on how to extend 
the numerical method and algorithm of this paper to a more general media cases.
We note that the boundary condition given in \eqref{Eq:ElasticBC} is known as the first 
order absorbing boundary condition (ABC) and this boundary condition 
simulates an unbounded domain by absorbing plane waves that come into the boundary in 
a normal direction (cf. \cite{Engquist_Majda_79}). We also note that since $\eta(\om,\cdot)$ is 
compactly supported on $D$ that $\alpha |_{\pa D} = 1$.  This choice was made to 
ensure that \eqref{Eq:ElasticBC} was indeed a first order ABC for every $\om \in \Om$.

Numerical approximations of random and stochastic partial differential equations (SPDEs) have
gained a lot of interests in recent years because of ever increasing needs
for modeling the uncertainties or noises that arise in industrial and engineering applications
\cite{Babuska_Nobile_Tempone_10, Babuska_Tempone_Zouraris_04,Caflisch_98, Fouque_Garnier_Papanicolaou_Solna_07, Ishimaru_97, 
Roman_Sarkis_06}. 
Two main numerical methods for random SPDEs are the Monte Carlo (finite element) 
method and the stochastic Galerkin method. The Monte Carlo method
obtains a set of independent identically distributed (i.i.d.) solutions by sampling
the PDE coefficients, and calculates the mean of the solution via a statistical
average over all the sampling in the probability space \cite{Caflisch_98}.
The stochastic Galerkin method, on the other hand, reduces the SPDE into a high dimensional 
deterministic equation by expanding the random coefficients in the equation using the 
Karhunen-Lo\`{e}ve or Wiener Chaos expansions \cite{Babuska_Nobile_Tempone_10,
Babuska_Tempone_Zouraris_04, Babuska_Tempone_Zouraris_05, DBO, Eiermann_Ernst_Ullmann_05, Liu_Riviere_13, 
Roman_Sarkis_06, Xiu_Karniadakis_02, Xiu_Karniadakis_02b}.
In general, these two methods become computationally expensive when a large number 
of degrees of freedom is involved in the spatial discretization, especially for 
Helmholtz-type equations. Indeed both methods become computationally prohibitive 
in the case that the frequency $k$ is large, because solving a
deterministic Helmholtz-type problem with large frequency is equivalent to
solving a large indefinite linear system of equations.  Furthermore, it is
well-known that standard iterative methods perform poorly when applied to linear
systems arising from Helmholtz-type problems \cite{Ernst_Gander_12}.
The Monte Carlo method requires solving the boundary value problem many times 
with different sampling coefficients, while the stochastic Galerkin method usually 
leads to a high dimensional deterministic equation that may be too expensive to solve.

Recently, we have developed a new efficient multi-modes Monte Carlo method for 
modeling acoustic wave propagation in weakly random media \cite{Feng_Lin_Lorton_15}. 
To solve the governing random Helmholtz equation,
the solution is first represented by a sum of mode functions, where each
mode satisfies a Helmholtz equation with deterministic coefficients and a random source.
The expectation of each mode function is then computed using a Monte Carlo interior penalty
discontinuous Galerkin (MCIP-DG) method.
We take advantage that the deterministic Helmholtz operators for all the modes are identical,
and employ an $LU$ solver for obtaining the numerical solutions.
Since the discretized equations for all the modes have the same constant coefficient matrix,
by using the $LU$ decomposition matrices repeatedly, the solutions for all samplings of mode functions
are obtained in an efficient way by performing simple forward
and backward substitutions. This leads to a tremendous saving in the computational costs.
Indeed, as discussed in \cite{Feng_Lin_Lorton_15}, the computational complexity
of the proposed algorithm is comparable to that of solving a few
deterministic Helmholtz problem using the a Gaussian elimination direct solver.
%
%
Due to the similarities between the scalar and elastic Helmholtz operators, it
is natural to extend the multi-modes MCIP-DG method of \cite{Feng_Lin_Lorton_15} to the 
elastic case for solving \eqref{Eq:ElasticPDE}--\eqref{Eq:ElasticBC}.  On 
the other hand the scalar and elastic Helmholtz operators have different
behaviors and kernel spaces so a separate study must be carried out to
construct and analyze the  multi-modes MCIP-DG method for the elastic Helmholtz problem.
This is exactly the primary goal of this paper.

The rest of the paper is organized as follows.  In Section \ref{sec:PDE_Analysis} we present
a complete PDE analysis of problem \eqref{Eq:ElasticPDE}--\eqref{Eq:ElasticBC}, including 
frequency-explicit solution estimates along with existence and uniqueness of solutions.  
In Section \ref{sec:Multi-Modes} the multi-modes expansion of the solution is defined and the
convergence of the expansion is also demonstrated. Moreover, error estimates are derived
for its finite term approximations. In Section \ref{sec:MCIP-DG} we formulate
our MCIP-DG method and derive error estimates for the method.  Section 
\ref{sec:Numerical_Procedure} lays out the overall multi-modes MCIP-DG 
algorithm. Computational complexity and convergence rate analysis are carried 
out for the algorithm. In Section \ref{sec:Numerical_Experiments} we
present several numerical experiments to demonstrate the performance and key features of
the proposed multi-modes MCIP-DG method and the overall algorithm.
Finally, in Section \ref{sec-7} we describe an idea on how to extend the proposed
MCIP-DG method and algorithm to the cases where more general (i.e., non-weak) random 
media must be considered.

\section{PDE analysis} \label{sec:PDE_Analysis}

\subsection{Preliminaries}
Standard function and space notations are adopted in this paper.  
$\bL^2(D) = \big( L^2(D) \big)^d$ denotes the space of all complex 
vector-valued square-integrable functions on $D$, and $\bH^s(D) = \big( H^s(D) 
\big)^d$ denotes the standard complex vector-valued Sobolev space.  For any 
$S \subset D$ and $\Sigma \subset \pa D$ we let $(\cdot,\cdot)_S$ and $
\langle \cdot,\cdot \rangle_\Sigma$ denote the standard complex-valued 
$L^2$-inner products on $S$ and $\Sigma$, respectively.  We also define the special 
function spaces
\begin{align*}
	\bH^1_+(D) &:= \Big\{ \bfv \in \bH^1(D) \, ; \, \bnabla \bfu \big|_{\pa 
	D} \in \bL^2(\pa D) \Big\}, \\
	\bV &:= \Big\{ \bfv \in \bH^1_+(D) \, ; \, \ddiv(\sigma (\bfu)) \in 
	\bL^2(D) \Big\}.
\end{align*}
Without loss of generality, we assume that the domain $D \subset 
B_R(\mathbf{0})$.  Throughout this paper we also assume that $D$ is a 
convex polygonal or a smooth domain that satisfies a star-shape condition 
with respect to the origin, i.e. there exists a positive constant $c_0$ such that 
\begin{align*}
	\bfx \cdot \bfnu \geq c_0 \qquad \mbox{ for all } \bfx \in \pa D.
\end{align*}
$\bL^2(\Om) = \big( L^2(\Om) \big)^d$ will denote the space of vector-valued 
square integrable functions on the probability space $(\Omega, \mathcal{F}, 
P)$.  $\E(\cdot)$ will denote the expectation operator given by
\begin{align} \label{eq:Expectation}
	\E(\bfv) := \int_{\Om} \bfv \, dP \qquad \mbox{ for all } \bfv \in 
	\bL^2(\Om),
\end{align}
and the abbreviation $a.s.$ will stand for \textit{almost surely}.

With these conventions in place, we introduce the following definition.
\begin{definition} \label{Def:WeakSolution}
Let $\bff \in \bL^2(\Om, \bL^2(D))$.  A function $\bfu \in \bL^2(\Om, 
\bH^1(D))$ is called a weak solution to problem 
\eqref{Eq:ElasticPDE}--\eqref{Eq:ElasticBC} if it satisfies the following identity:
\begin{align} \label{Eq:WeakFormulation}
	\int_\Om a(\bfu, \bfv) \, dP = \int_\Om (\bff, \bfv)_D \, dP \qquad 
	\forall \bfv \in \bL^2(\Om, \bH^1(D)),
\end{align}
where
\begin{align} \label{Eq:WeakSesquilinearForm}
	a(\bfw, \bfv) &:= 2\mu (\bsnabla\bfw,\bsnabla\bfv)_D + \lambda (\ddiv \bfw, \ddiv 
	\bfv)_D - k^2( \alpha^2 \bfw, \bfv)_D \\
	& \qquad + \i k \langle A \bfw, \bfv \rangle_{\pa D}. \notag
\end{align}
\end{definition}
To simplify the analysis throughout the rest of this paper we introduce the 
following special semi-norm on $\bH^1(D)$:
\begin{align} \label{eq:Seminorm1}
	| \bfv |_{1,D} := \lambda \| \ddiv \bfv \|_{L^2(D)} + 2 \mu \| \bsnabla\bfv \|
	_{L^2(D)} \qquad \mbox{ for all } \bfv \in \bH^1(D).
\end{align}
\begin{remark}
	a) Korn's second inequality ensures that the semi-norm $| \cdot |_{1,D}$ 
	defined above is equivalent to the standard $H^1$ semi-norm.
	
	b) By using Lemma \ref{lem:PDEEstimate3} below, it is easy to show that 
	any solution $\bfu$ of \eqref{Eq:WeakFormulation}--
	\eqref{Eq:WeakSesquilinearForm} satisfies $\bfu \in \bL^2( \Om, \bV)$.
\end{remark}

In \cite{Cummings_Feng_06}, it was shown that estimates for solutions of the 
deterministic elastic Helmholtz problem are optimal in $k$ when the solution 
satisfies a Korn-type inequality on the boundary of the form
\begin{align*}
	\| \bnabla \bfu \|^2_{L^2(\pa D)} \leq \tilde{K} \left[ \| \bnabla \bfu \|
	^2_{L^2(D)} + \| \bsnabla \bfu \|^2_{L^2(\pa D)} + \| \ddiv \bfu \|
	^2_{L^2(\pa D)} \right],
\end{align*}
where $\tilde{K}$ is a positive constant independent of $\bfu$.
In the next subsection, a similar result is shown for solutions satisfying a 
stochastic Korn-type inequality on the boundary given by
\begin{align} \label{eq:Stoch_Bound_Korn}
	&\E \big( \| \bnabla \bfu \|^2_{L^2(\pa D)} \big) \\
	& \qquad \leq \tilde{K} \left[ \E 
	\big( \| \bnabla \bfu \|^2_{L^2(D)} \big) + \E \big( \| \bsnabla \bfu \|
	^2_{L^2(\pa D)} \big) + \E \big( \| \ddiv \bfu \|^2_{L^2(\pa D)} \big) 
	\right]. \notag
\end{align}
The Korn-type inequality above is just a conjecture at this point.
For this reason we introduce the special function space
\begin{align*}
	\bV_{\tilde{K}} &:= \Big\{ \bfv \in \bH^1_+(D) \, ; \, \bfv \, 
	\mbox{satisfies \eqref{eq:Stoch_Bound_Korn}} \, \Big\},
\end{align*}
for some $\tilde{K}$ independent of $\bfu$.
We also introduce a special parameter $\ta$, that will be used in the 
solution estimates presented in the next section.  $\ta$ is defined as
\begin{align} \label{eq:ta}
	\ta := \left\{
	\begin{array}{l r}
	1 & \mbox{ if } \bfu \in \bV_{\tilde{K}} \\
	2 & \mbox{ otherwise }
	\end{array}
	\right. ,
\end{align}
where $\bfu$ is the solution to \eqref{Eq:ElasticPDE}--\eqref{Eq:ElasticBC}.

\subsection{Frequency-explicit solution estimates}
In this subsection, we derive stability estimates for the solution 
of problem \eqref{Eq:WeakFormulation}--\eqref{Eq:WeakSesquilinearForm}.  
Since the main concern of this paper is the case 
when $k$ is large, we make the assumption that $k \geq 1$ to 
simplify some of the estimates.  The goal of this section is to derive solution
estimates that are explicitly dependent on the frequency $k$.  
These frequency-explicit estimates play a pivotal role in the 
development of numerical methods for deterministic wave equations (cf. 
\cite{Feng_Wu_09}, \cite{Feng_Wu_11},\cite{Feng_Lorton_15}). We 
will obtain existence and uniqueness of solutions to 
\eqref{Eq:WeakFormulation}--\eqref{Eq:WeakSesquilinearForm} as a direct 
consequence of the estimates established in this subsection.  

We begin with a number of technical lemmas which will be used in the proof 
of our solution estimates. Our analysis follows the 
analysis for the deterministic elastic Helmholtz equations carried out in 
\cite{Cummings_Feng_06} with many changes made to accommodate the inclusion 
of the random field $\alpha(\om,\cdot)$. 

For many of the estimates derived in this paper, it is important to note that 
the matrix $A$ in \eqref{Eq:ElasticBC} is a real symmetric positive definite 
matrix.  Thus, there exist positive constants $c_A$ and $C_A$ such that
\begin{align} \label{eq:SPDProperty}
	c_A \| \bfu \|^2_{L^2(\pa D)} \leq \langle A \bfu, \bfu \rangle_{\pa D} 
	\leq C_A \| \bfu \|^2_{L^2(\pa D)},
\end{align}
for all $\bfu \in \bL^2(\pa D)$.

\begin{lemma} \label{lem:PDEEstimate1}
	Suppose $\bfu \in \bL^2(\Om, \bH^1(D))$ solves 
	\eqref{Eq:WeakFormulation}--\eqref{Eq:WeakSesquilinearForm}.  
	Then for any $\delta_1, \delta_2 > 0$, $\bfu$ satisfies the following 
	estimates:
	\begin{align} \label{eq:PDEEstimate1_1}
		\E \big( | \bfu |^2_{1,D} \big) &\leq \big((1+\veps)^2 k^2 + 
		\delta_1\big) \E \big( \| \bfu \|^2_{L^2(D)} \big) + \frac{1}
		{4\delta_1} \E\big( \| \bff \|^2_{L^2(D)} \big), \\
		\E \big( \| \bfu \|^2_{L^2(\pa D)} \big) & \leq \frac{\delta_2}{c_A 
		k} \E\big( \| \bfu \|^2_{L^2(D)} \big) + \frac{1}
		{4\delta_2 c_A k} \E \big( \| \bff \|^2_{L^2(D)} \big). 
		\label{eq:PDEEstimate1_2} 
	\end{align}
\end{lemma}

\begin{proof}
	Setting $\bfv = \bfu$ in \eqref{Eq:WeakFormulation} and taking the real 
	and imaginary parts separately yields
	\begin{align} \label{eq:PDEEstimate1_3}
		&\re \int_\Om (\bff, \bfu)_D  \, dP = \int_\Om | \bfu |^2_{1,D} - k^2 
		\big \| (1 + \veps \eta) \bfu \big \|^2_{L^2(D)} \, dP, \\
		&\im \int_\Om  (\bff, \bfu)_D \, dP \label{eq:PDEEstimate1_4}
		= k \int_\Om \big \langle A \bfu, \bfu \big \rangle_{\pa D} \, dP .
	\end{align}
	\eqref{eq:PDEEstimate1_1} is obtained by rearranging the terms in 
	\eqref{eq:PDEEstimate1_3} and applying the Cauchy-Schwarz inequality.
Applying the Cauchy-Schwarz inequality along with \eqref{eq:SPDProperty} to 
\eqref{eq:PDEEstimate1_4} produces
\begin{align*}
	c_A k \E \big( \| \bfu \|^2_{L^2(\pa D)} \big) & \leq \delta_2 
	\E \big( \| \bfu \|^2_{L^2(D)} \big) + \frac{1}{4 \delta_2} \E \big( \| 
	\bff \|^2_{L^2(D)} \big).
\end{align*}
We divide both sides of the above inequality by $c_A k$. This yields 
\eqref{eq:PDEEstimate1_2}.  The proof is complete.
\end{proof}

Korn's second inequality is essential to establishing solution 
estimates for the deterministic elastic Helmholtz equation (cf. 
\cite{Cummings_Feng_06}).  Here we state a stochastic version of this 
inequality.
\begin{lemma}[Stochastic Korn's Inequality] \label{lem:Korn}
	Let $\bfu \in \bL^2(\Om,\bH^1(D))$, then there exists a positive constant 
	$K$ such that
	\begin{align}
		\E \big( \| \bsnabla\bfu \|^2_{L^2(D)} \big) + \E \big( \|  \bfu \|^2
		_{L^2(D)} 
		\big) \geq K \E \big( \| \bfu \|^2_{H^1(D)} \big). \label{eq:Korn}
	\end{align}
\end{lemma}
\begin{proof}
	There exists $K > 0$ such that for fixed $\om \in \Om$, $\bfu(\om,\cdot)$ 
	satisfies Korn's second inequality
	\begin{align*}
		\big\| \bsnabla\bfu(\om,\cdot) \big\|^2_{L^2(D)}  +  \big \|  \bfu(\om,
		\cdot) \big \|^2_{L^2(D)}  \geq K \big \| \bfu(\om,\cdot) \big \|^2
		_{H^1(D)} .
	\end{align*}
For a proof of this inequality see \cite{Nitsche_98}.
Integrating over $\Om$ yields \eqref{eq:Korn}.  The proof is complete.
\end{proof}

Rellich identities for the elastic Helmholtz operator are also essential in 
the proof of solution estimates for the deterministic elastic 
Helmholtz equation (cf. \cite{Cummings_Feng_06}).  Here we state stochastic 
versions of these Rellich identities.
\begin{lemma}[Stochastic Rellich Identities] \label{lem:Rellich}
	Suppose that $\bfu \in \bL^2(\Om,\bH^2(D))$.  Then for $\bfv(\om,\bfx) 
	:= (\bnabla \bfu(\om,\bfx)) \bfx$, we have the following stochastic 
	Rellich identities:
	\begin{align}
		&\re \int_\Om (\bfu, \bfv)_D \, dP = \int_\Om \Big( -\frac{d}{2} \| 
		\bfu \|^2_{L^2(D)} + \frac{1}{2} \big \langle \bfx \cdot \bfnu, | \bfu 
		|^2 \big \rangle_{\pa \Om} \Big) \, dP, \label{eq:Rellich1} \\
		&\re \int_\Om \Big( 2\mu (\bsnabla\bfu, \bsnabla\bfv)_D + \lambda 
		(\ddiv \bfu,	 \ddiv \bfv)_D \Big) \, dP \label{eq:Rellich2} \\
		& \qquad = \frac{2-d}{2} \int_\Om \Big( 2\mu \|\bsnabla\bfu\|
		^2_{L^2(D)} + 
		\lambda \|\ddiv \bfu \|^2_{L^2(D)} \Big) \, dP \notag \\
		& \qquad \qquad + \frac{1}{2} \int_\Om \Big( 2\mu \big \langle \bfx 
		\cdot \bfnu, | \ddiv \bfu |^2 \big \rangle_{\pa \Om} + \lambda \big 
		\langle \bfx \cdot \bfnu, |\bsnabla\bfu|^2 \big \rangle_{\pa \Om} 
		\Big) \, dP. \notag
	\end{align}
\end{lemma}
\begin{proof}
	For any $\om \in \Om$ we obtain the following identities from Proposition 
	2 and Lemma 5 of \cite{Cummings_Feng_06}:
	\begin{align*}
		&\re \big(\bfu(\om,\cdot), \bfv(\om,\cdot)\big)_D  = -\frac{d}{2} 
		\big \| \bfu(\om,\cdot) \big \|^2_{L^2(D)} + \frac{1}{2} \big \langle 
		\bfx \cdot \bfnu, | \bfu(\om,\cdot) |^2 \big \rangle_{\pa \Om}, \\
		&\re \Big( 2\mu \big ( \bsnabla\bfu(\om,\cdot), \bsnabla\bfv(\om,
		\cdot) \big )_D + \lambda \big ( \ddiv \bfu(\om,\cdot), \ddiv \bfv(\om,
		\cdot) \big )_D \Big) \\
		& \qquad = \frac{2-d}{2} \Big( 2\mu \big \|\bsnabla\bfu(\om,\cdot) 
		\big \|^2_{L^2(D)} + \lambda \big \|\ddiv \bfu(\om,\cdot) \big \|
		^2_{L^2(D)} \Big)  \\
		& \qquad \qquad + \frac{1}{2} \Big( 2\mu \big \langle \bfx \cdot 
		\bfnu, | \ddiv \bfu(\om,\cdot) |^2 \big \rangle_{\pa \Om} + \lambda 
		\big \langle \bfx \cdot \bfnu, |\bsnabla\bfu(\om,\cdot)|^2 \big 
		\rangle_{\pa \Om} \Big) .
	\end{align*}
\eqref{eq:Rellich1} and \eqref{eq:Rellich2} are obtained by integrating the 
above identities over $\Om$. The proof is complete.
\end{proof}

The following two lemmas relate higher order norms of the solution $\bfu$ of 
\eqref{Eq:WeakFormulation}--\eqref{Eq:WeakSesquilinearForm} to the $\bL^2$-norm 
of $\bfu$.
\begin{lemma} \label{lem:PDEEstimate2}
	Suppose that $\bfu \in \bL^2(\Om, \bH^1(D))$ solves 
	\eqref{Eq:WeakFormulation}--\eqref{Eq:WeakSesquilinearForm}. Then 
	for all $\delta > 0$,
	\begin{align} \label{eq:PDEEstimate2_1}
		&2 \mu K \E \big( \| \bfu \|_{H^1(D)}^2 \big) \\
		& \qquad \leq 2 \left( k^2(1 + 
		\veps)^2 + \delta + 2\mu \right) \E \big( \| \bfu \|_{L^2(D)}^2 \big) 
		+  \frac{1}{2 \delta}  \E\big( \| \bff \|^2_{L^2(D)} \big). \notag
	\end{align}
\end{lemma}
\begin{proof}
	We obtain \eqref{eq:PDEEstimate2_1} by combining 
	\eqref{eq:PDEEstimate1_1} and \eqref{eq:Korn}.
\end{proof}

\begin{lemma} \label{lem:PDEEstimate3}
	Suppose that $\bfu \in \bL^2( \Om, \bH^2(D))$ solves 
	\eqref{Eq:WeakFormulation}--\eqref{Eq:WeakSesquilinearForm} with $k \geq 1$ and $0 \leq \veps < 1$. 
	Then there exists a positive constant $C$, independent of $\veps$ 
	and $k$ such that
	\begin{align} \label{eq:PDEEstimate3_1}
		&\E \big(\| \bfu \|^2_{H^2(D)} \big) \leq C \left(1 + (1+\veps)^2k^2 
		\right)^2 \E \big( \| \bfu \|^2_{L^2(D)} \big) + C \E \big( \| \bff \|
		^2_{L^2(D)} \big), \\
		&\E \big( \| \bnabla \bfu \|^2_{L^2(\pa D)} \big)  
		\label{eq:PDEEstimate3_2} 
		\leq C \left( k^3 \E \big( \| \bfu \|^2_{L^2(D)} \big) + \frac{1}{k} 
		\E \big( \| \bff \|^2_{L^2(D)} \big) \right).
	\end{align}
	\end{lemma}
	\begin{proof}
		Regularity theory for elliptic problems 
		\cite{Grisvard_92,Ladyzenskaja_Solonnikov_Uralceva_68} implies for 
		a.e. $\om \in \Om$
		\begin{align*}
			\| \bfu(\om,\cdot) \|^2_{H^2(D)} & \leq C \Big( \| \bff(\om,
			\cdot) \|^2_{L^2(D)} + \big( (1+\veps)^2 k^2 \big)^2 \| \bfu(\om,
			\cdot) \|^2_{L^2(D)} \\
			& \qquad \qquad + C_A k^2 \| \bfu(\om,\cdot) \|
			^2_{H^{\frac{1}{2}}(\pa D)} + \| \bfu(\om,\cdot) \|^2_{L^2(D)} 
			\Big) \\
			&\leq C \Big( \| \bff(\om,\cdot) \|^2_{L^2(D)} + \big(1 + (1+
			\veps)^2 k^2 \big)^2 \| \bfu(\om,\cdot) \|^2_{L^2(D)} \\ 
			&  \qquad \qquad + k^2 \| \bfu(\om,\cdot) \|^2_{H^{1}
			(D)}\Big).
		\end{align*}
		Taking the expectation on both sides and 
		using Lemma \ref{lem:PDEEstimate2} yield
		\begin{align*}
			\E \big(\| \bfu \|^2_{H^2(D)} \big) &  \leq C \Big( \E \big(\| 
			\bff \|^2_{L^2(D)} \big)+ \big(1 + (1+\veps)^2 k^2 \big)^2 \E 
			\big( \| \bfu \|^2_{L^2(D)} \big) \Big) \\
			& \qquad + \frac{C k^2}{\mu K} \left( \left( k^2 (1 
			+ \veps)^2 + \delta + 2 \mu \right) \E \big( \| \bfu \|
			^2_{L^2(D)} \big) + \frac{1}{4 \delta} \big( \| \bff \|
			^2_{L^2(D)} \big) \right).
		\end{align*}
		Hence \eqref{eq:PDEEstimate3_1} holds with $\delta = (1+\veps)^2 k^2$.
		
		To prove \eqref{eq:PDEEstimate3_2}, we note that $\pa D$ is piecewise 
		smooth.  Thus, by the trace inequality, \eqref{eq:PDEEstimate2_1}, 
		and \eqref{eq:PDEEstimate3_1} the following inequalities hold
		\begin{align*}
		\E \big( \| \bnabla \bfu \|^2_{L^2(\pa D)} \big) & \leq C \int_\Om \| 
		\bnabla \bfu \|_{L^2(D)} \| \bfu \|_{H^2(D)} \, dP \\
		& \leq C k \E \big( \| \bnabla \bfu \|^2_{L^2(D)} \big) + \frac{C}{k} 
		\E \big( \| \bfu \|^2_{H^2(D)} \big) \\
		& \leq \frac{C k}{\mu K} \left( \left( (1 + \veps)^2 k^2 + \delta + 
		2\mu \right) \E \big( \| \bfu \|_{L^2(D)}^2 \big)  +  \frac{1}{4 
		\delta}  \E\big( \| \bff \|^2_{L^2(D)} \big) \right) \\
		& \qquad + \frac{C}{k} \Big( \big( 1 + (1 + \veps)^2 k^2 \big)^2 \E 
		\big( \| \bfu \|^2_{L^2(D)} \big) + \E \big( \| \bff \|^2_{L^2(D)} 
		\big) \Big).
		\end{align*}
		Letting $\delta = (1+\veps)^2 k^2$ in the above inequality yields
		\begin{align*}
			\E \big( \| \bnabla \bfu \|^2_{L^2(\pa D)} \big) & \leq C \left( 
			k \left(1 + (1 + \veps)^2 k^2 \right) + \frac{1}{k} \left(1 + (1+ 
			\veps)^2 k^2 \right)^2 \right) \E \big( \| \bfu \|^2_{L^2(D)} 
			\big) \\
			& \qquad + C \left( \frac{1}{(1+\veps)^2k} + \frac{1}{k} \right)
			\E \big( \| \bff \|^2_{L^2(D)} \big).
		\end{align*}
		By the assumptions $0 \leq \veps < 1$ and $k \geq 1$ we obtain
		\eqref{eq:PDEEstimate3_2}. The proof is complete.	
	\end{proof}
	
With these technical lemmas in place we are now ready to prove the main 
result for this section.

\begin{theorem} \label{thm:PDEEstimate}
Suppose that $\bfu \in \bL^2( \Om, \bV)$ solves 
\eqref{Eq:WeakFormulation}--\eqref{Eq:WeakSesquilinearForm} with $k \geq 1$. 
Further, let $D$ be a convex polygonal or a smooth domain and $R$ be the smallest 
number such that $D \subset B_R(\mathbf{0})$. Then the following estimates hold
\begin{align}\label{eq:PDEEstimate4_1}
	&\E \left( \| \bfu \|^2_{L^2(D)} + \| \bfu \|_{L^2(\pa D)}^2 + \frac{c_0}
	{k^2} | \bfu |^2_{1,\pa D} \right)\\
	& \hspace*{2cm} \leq C_0 \left(k^{\ta - 2} + \frac{1}
	{k^2} \right)^2 \E \big( \| \bff \|^2_{L^2(D)} \big), 
	\notag \\
	&\E \big( \| \bfu \|^2_{H^1(D)} \big) \leq C_0 \left( k^{\ta - 1} + 
	\frac{1}{k^2} \right)^2 \E \big( \| \bff \|^2_{L^2(D)} \big), 
	\label{eq:PDEEstimate4_2}
\end{align}
provided that $\veps (2 + \veps) < \gamma_0 := \min \left \{1, \frac{1}{4} 
\left( d-1 + \frac{3kR}{\mu K} + \frac{2k R}{K} + kR
\right)^{-1} \right \}$.  Here $C_0$ is a positive constant independent of $k$ and 
$\bfu$, and $\ta$ is defined by \eqref{eq:ta}.
Moreover, if $\bfu \in \bL^2(\Om, \bH^2(D))$ the following estimate 
also holds
\begin{align}
	\E \big( \| \bfu \|^2_{H^2(D)} \big) \leq C_0 \left( k^{\ta} + \frac{1}
	{k^2} \right)^2 \E \big( \| \bff \|^2_{L^2(D)} \big). 
	\label{eq:PDEEstimate4_3}
\end{align}
\end{theorem}

\begin{proof}
{\em Step 1:} We  begin by proving \eqref{eq:PDEEstimate4_1}.
Let $\bfv = (\bnabla \bfu) \bfx$ in \eqref{Eq:WeakFormulation}.  By taking 
the real part and rearranging terms we find
\begin{align*}
	& \re \int_\Omega \Big( 2 \mu \big( \bsnabla\bfu, \bsnabla\bfv \Big)_D + 
	\lambda 	\big( \ddiv \bfu, \ddiv \bfv \big)_D - k^2 \big( \bfu, \bfv 
	\big)_{D} 	\Big) \, dP \\
	& \qquad \leq \re \int_\Omega \Big( k^2 \veps \big( \eta (2 + \veps \eta)
	\bfu, \bfv \big)_D + \big( \bff, \bfv \big)_D \Big) \, dP + \im \int_
	\Omega k \big \langle A \bfu, \bfv \big \rangle_{\pa D} \, dP.
\end{align*}
To this identity we apply the Rellich identities in \eqref{eq:Rellich1} and 
\eqref{eq:Rellich2},  after another rearrangement of terms we get
\begin{align*}
	& \frac{d k^2}{2} \E \big( \| \bfu \|^2_{L^2(D)} \big) \\ 
	& \qquad \leq \frac{1}{2} \int_\Om \Big( k^2 \big \langle \bfx \cdot 
	\bfnu, | \bfu |^2 \big \rangle_{\pa \Om} - 2 \mu \big \langle \bfx \cdot 
	\bfnu, |\bsnabla\bfu|^2 \big \rangle_{\pa \Om} - \lambda \big \langle \bfx 
	\cdot \bfnu, | \ddiv \bfu |^2 \big \rangle_{\pa \Om} \Big) \, dP \\
	& \qquad \qquad + \frac{d - 2}{2} \E \big( | \bfu |^2_{1,D} \big) +  \re 
	\int_\Om \Big( k^2 \veps \big( \eta(2 + \veps \eta) \bfu, \bfv \big)_D + 
	\big( \bff, \bfv \big)_D \Big) \, dP \\
	& \qquad \qquad + \im \int_\Omega k\big \langle  A \bfu, 
	\bfv \big \rangle_{\pa D} \, dP. \\
\end{align*}
Using the fact that $D$ is star-shaped and $D \subset B_R(\mathbf{0})$ 
along with the Cauchy-Schwarz inequality we obtain
\begin{align} \label{eq:PDEEstimate4_aa}
	& \frac{d k^2}{2} \E \big( \| \bfu \|^2_{L^2(D)} \big) \\
	& \qquad \leq \frac{k^2 R}{2} \E \big( \| \bfu \|^2_{L^2(\pa D)} \big) - 
	\frac{c_0}{2} \E \big( | \bfu |^2_{1,\pa D} \big) + \frac{d - 2}{2} \E 
	\big( | \bfu |^2_{1,D} \big) \notag \\
	& \qquad \qquad + k^2 R \veps (2 + \veps)  \left( \frac{1}{2 \delta_1} \E 
	\big( \| \bfu \|^2_{L^2(D)} \big) + \frac{\delta_1}{2} \E \big( \| 
	\bnabla \bfu \|^2_{L^2(D)} \big) \right) \notag \\
	& \qquad \qquad + \frac{R}{2 \delta_2} \E \big( \| \bff \|^2_{L^2(D)} 
	\big) + \frac{R \delta_2}{2} \E \big( \| \bnabla \bfu \|^2_{L^2(D)} \big) \notag \\
	& \qquad \qquad + \frac{k C_A R}{2 \delta_3} \E \big( \| \bfu \|
	^2_{L^2(\pa D)} \big) + \frac{k C_A R \delta_3}{2} \E \big( \| \bnabla 
	\bfu \|^2_{L^2(\pa D)} \big). \notag
\end{align}
From \eqref{eq:PDEEstimate1_1} we find 
\begin{align} \label{eq:PDEEstimate4_aaa}
	\E \big( | \bfu |^2_{1,D} \big) - k^2 \E \big( \| \bfu \|^2_{L^2(D)} \big) 
	\leq \big( k^2 \veps(2 + \veps) + \delta_4 \big) \| \bfu \|^2_{L^2(D)} + 
	\frac{1}{4 \delta_4} \| \bff \|^2_{L^2(D)}.
\end{align}
By adding $\frac{d-1}{2}$ times \eqref{eq:PDEEstimate4_aaa} to 
\eqref{eq:PDEEstimate4_aa}, grouping like terms and letting $\gamma := \veps(2+
\veps)$, we get
\begin{align}
	\label{eq:PDEEstimate4_a} & \frac{k^2}{2} \E \big( \| \bfu \|^2_{L^2(D)} 
	\big) + \frac{c_0}{2} \E \big( | \bfu |^2_{1,\pa D} \big) + \frac{1}{2} \E 
	\big( | \bfu |^2_{1,D} \big) \\
	& \qquad \leq \frac{k^2 \gamma R}{2 \delta_1} \E \big( \| \bfu \|^2_{L^2(D)} 
	\big) + \left( \frac{k^2 \gamma R \delta_1}{2} + \frac{R \delta_2}{2} 
	\right) \E \big( \| \bnabla \bfu \|^2_{L^2(D)} \big) \notag \\
	& \qquad \qquad + \left( \frac{k^2 R}{2} + \frac{k C_A R}{2 \delta_3} 
	\right) \E \big( \| \bfu \|^2_{L^2(\pa D)} \big) + \frac{k C_A \delta_3}{2} 
	\E \big( \| \bnabla \bfu \|^2_{L^2(\pa D)} \big) \notag \\
	& \qquad \qquad + \big( k^2 \gamma + \delta_4 \big) \| \bfu \|
	^2_{L^2(D)} + \frac{1}{4 \delta_4} \| \bff \|^2_{L^2(D)} + \frac{R}{2 
	\delta_2} \E \big( \| \bff \|^2_{L^2(D)} \big). \notag
\end{align}

\medskip
{\em Step 2:}
The source of the different values of $\ta$ in \eqref{eq:PDEEstimate4_1} comes 
from different treatments for the $\E \big( \| \bnabla \bfu \|^2_{L^2(\pa D)} \big)$ 
term.  In particular, if $\bfu \in \bV_{\tilde{K}}$ we apply 
\eqref{eq:Stoch_Bound_Korn} to control this term. Otherwise, we apply 
\eqref{eq:PDEEstimate3_2} to control this term.  We first prove 
\eqref{eq:PDEEstimate4_1} with $\ta = 2$.
Applying \eqref{eq:PDEEstimate1_2},
\eqref{eq:PDEEstimate2_1} and \eqref{eq:PDEEstimate3_2} to 
\eqref{eq:PDEEstimate4_a} yields.
\begin{align*}
	&  \frac{k^2}{2} \E \big( \| \bfu \|^2_{L^2(D)} 
	\big) + \frac{c_0}{2} \E \big( | \bfu |^2_{1,\pa D} \big) + \frac{1}{2} \E 
	\big( | \bfu |^2_{1,D} \big) \\
	&  \leq \frac{1}{\mu K} \left( \frac{k^2 \gamma R \delta_1}{2}+ \frac{R 
	\delta_2}{2} \right) \left( \big(k^2(1 + \gamma) + \delta_5 + 2\mu \big) 
	\E \big( \| \bfu \|^2_{L^2(D)} \big) + \frac{1}{2 \delta_5} \E \big( \| 
	\bff \|^2_{L^2(D)}\big) \right) \\
	& \qquad + \left( \frac{k^2 R}{2} + \frac{k C_A R}{2 \delta_3} \right) 
	\left( \frac{\delta_6}{c_A k} \E \big( \| \bfu \|^2_{L^2(D)} 
	\big)  + \frac{1}{4 \delta_6 c_A k}\E \big( \| \bff \|
	^2_{L^2(D)}\big) \Big) \right) \\
	& \qquad + \frac{C k C_A R \delta_3}{2} \left( k^3 \E \big( \| \bfu \|
	^2_{L^2(D)} \big) + \frac{1}{k} \E \big( \| \bff \|^2_{L^2(D)} \big) 
	\right) \\
	& \qquad + \frac{d-1}{2} \left( \big(\gamma k^2 + \delta_4 \big)\E 
	\big( \| \bfu \|^2_{L^2(D)} \big) + \frac{1}{4 \delta_4}\E \big( \| \bff 
	\|^2_{L^2(D)}\big) \right) \\
	& \qquad + \frac{k^2 \gamma R}{2 \delta_1} \E \big( \| \bfu \|^2_{L^2(D)} 
	\big) + \frac{R}{2 \delta_2} \E \big( \| \bff \|^2_{L^2(D)} \big).
\end{align*}
In order to apply \eqref{eq:PDEEstimate3_2} we require $\veps < 1$.  Since 
$ \gamma \leq 1$, it can be easily shown that $\veps \leq \frac{1}{2}$.  Thus,
\begin{align} \label{eq:PDEEstimate4_4}
	c_1 \E \big( \| \bfu \|^2_{L^2(D)} \big) + \frac{c_0}{2} \E \big( | \bfu 
	|^2_{1,\pa D} \big) + \frac{1}{2} \E \big( | \bfu |^2_{1,D} \big) \leq c_2 
	\E \big( \| \bff \|^2_{L^2(D)} \big),
\end{align}
where
\begin{align*}
	c_1 &:= \frac{k^2}{2} - \frac{d-1}{2} \big( \gamma k^2 + \delta_4 \big) - \frac{1}
	{\mu K} \left( \frac{k^2 \gamma R \delta_1}{2} + \frac{R \delta_2}{2} 
	\right) \big(k^2(1 + \gamma) + \delta_5 + 2\mu k^2 \big) \\
	& \qquad - \left( \frac{k^2 R}{2} + \frac{k C_A R}{2 \delta_3} \right) 
	\frac{\delta_6}{c_A k} - \frac{C C_A k^4 R \delta_3}{2} - \frac{k^2 
	\gamma R}{2 \delta_1}, \\
\end{align*}
and
\begin{align*}
	c_2 &:= \frac{1}{2 \mu K \delta_5} \left( \frac{k^2 \gamma R \delta_1}{2}
	+ \frac{R \delta_2}{2} \right) + \frac{1}{4 \delta_6 c_A k} 
	\left( \frac{k^2 R}{2} + \frac{k C_A R}{2 \delta_3} \right) \\
	& \qquad + \frac{C k C_A R \delta_3}{2k} + \frac{d-1}{8 
	\delta_4} + \frac{R}{2 \delta_2}.
\end{align*}
In the third term of $c_1$, we have used the fact that $k \geq 1$ to include a coefficient 
of $k^2$ to the $2\mu$ term.  This has been done to simplify constants later. Setting
\begin{align*}
	\begin{array}{l l l}
	\delta_1 = \frac{1}{2k}, & \delta_2 = \frac{\mu K}{16R(3 + 2 \mu)}, 
	& \delta_3 = \frac{1}{8C C_A k^2 R}, \\[.5cm]
	\delta_4 = \frac{1}{16(d-1)}, & \delta_5 = \frac{k^2}{2}, & \delta_6 = 
	\frac{c_A k^2}{8\big(kR + 8 C C_A^2 k^2 R^2 
	\big)},
	\end{array}
\end{align*}
and using the fact that $\gamma \leq 1$ yields
\begin{align*}
	c_1 = \frac{k^2}{2} - \frac{k^2}{2} \left( d-1 + \frac{3kR}{\mu K} + 
	\frac{2kR}{K} + kR \right) \gamma - \frac{k^2}{4}.
\end{align*}
Also using the fact that $\gamma \leq \frac{1}{4} \left( d-1 + \frac{3kR}{\mu K} 
+ \frac{2kR}{K} + kR \right)^{-1}$ yields $c_1 \geq 
\frac{k^2}{8}$.  It is easy to check that $c_2 \leq C \left(k^2 + \frac{1}
{k^2} \right)$. Therefore, \eqref{eq:PDEEstimate4_4} becomes
\begin{align*}
	\frac{k^2}{8} \E \big( \| \bfu \|^2_{L^2(D)} \big) + \frac{c_0}{2} \E 
	\big( | \bfu |^2_{1,\pa D} \big) \leq C \left(k^2 + \frac{1}{k^2} \right) 
	\E \big( \| \bff \|^2_{L^2(D)} \big).
\end{align*}
Multiplying both sides by $\frac{8}{k^2}$ and applying 
\eqref{eq:PDEEstimate1_2} with $\delta_2 = k$ implies 
\eqref{eq:PDEEstimate4_1} with $\ta = 2$.

\medskip
{\em Step 3:}
If $\bfu \in \bV_{\tilde{K}}$, we apply \eqref{eq:Korn} and obtain.
\begin{align*}
	&\frac{k C_A}{2} \delta_3 \E \big( \| \bnabla \bfu \|^2_{L^2(\pa D} \big) - 
	\frac{1}{4} \left( k^2 \E \big( \| \bfu \|^2_{L^2(D)} \big) + c_0 \E \big( 
	| \bfu |^2_{1,\pa D}\big) + \E \big( | \bfu |^2_{1,D} \big) \right) \\
	& \qquad \leq \frac{k C_A}{2} \delta_3 \E \big( \| \bnabla \bfu \|
	^2_{L^2(\pa D} \big) - \frac{1}{4} \min \{ k^2K, 2 \mu K, c_0 \lambda, 2c_0 
	\mu \} \\
	& \qquad \qquad \cdot \Big( \E \big( \| \bnabla \bfu \|^2_{L^2(D)} \big) + 
	\E \big( \| \bsnabla \bfu \|^2_{L^2(\pa D)}\big) + \E \big( \| \ddiv \bfu 
	\|^2_{L^2(\pa D)} \big) \Big).
\end{align*} 
By choosing $\delta_3 = \frac{1}{2kC_A} \min \{ K, 2 \mu K , c_0 \lambda, 2c_0 
\mu \}$, using the fact that $k \geq 1$, and applying 
\eqref{eq:Stoch_Bound_Korn} we find
\begin{align*}
	\frac{k C_A}{2} \delta_3 \E \big( \| \bnabla \bfu \|^2_{L^2(\pa D} \big) - 
	\frac{1}{4} \left( k^2 \E \big( \| \bfu \|^2_{L^2(D)} \big) + c_0 \E \big( 
	| \bfu |^2_{1,\pa D}\big) + \E \big( | \bfu |^2_{1,D} \big) \right) \leq 0.
\end{align*}
We apply this inequality to \eqref{eq:PDEEstimate4_a} as well as 
\eqref{eq:PDEEstimate1_2} and \eqref{eq:PDEEstimate2_1} in order to find
\begin{align*}
	& \frac{k^2}{4} \E \big( \| \bfu \|^2_{L^2(D)} \big) + \frac{c_0}{4} \E 
	\big( | \bfu |^2_{1,\pa D} \big) + \frac{1}{4} \E \big( | \bfu |^2_{1,D} 
	\big) \\
	&  \leq \frac{1}{\mu K} \left( \frac{k^2 \gamma R \delta_1}{2}+ \frac{R 
	\delta_2}{2} \right) \left( \big(k^2(1 + \gamma) + \delta_5 + 2\mu \big) 
	\E \big( \| \bfu \|^2_{L^2(D)} \big) + \frac{1}{2 \delta_5} \E \big( \| 
	\bff \|^2_{L^2(D)}\big) \right) \\
	& \qquad + \left( \frac{k^2 R}{2} + \frac{k C_A R}{2 \delta_3}\right) 
	\left( \frac{\delta_6}{c_A k} \E \big( \| \bfu 
	\|^2_{L^2(D)} \big)  + \frac{1}{4 \delta_6 c_A k}\E \big( \| \bff \|
	^2_{L^2(D)}\big) \Big) \right) \\
	& \qquad + \frac{d-1}{2} \left( \big(\gamma k^2 + \delta_4 \big)\E 
	\big( \| \bfu \|^2_{L^2(D)} \big) + \frac{1}{4 \delta_4}\E \big( \| \bff 
	\|^2_{L^2(D)}\big) \right) \\
	& \qquad + \frac{k^2 \gamma R}{2 \delta_1} \E \big( \| \bfu \|^2_{L^2(D)} 
	\big) + \frac{R}{2 \delta_2} \E \big( \| \bff \|^2_{L^2(D)} \big).
\end{align*}
Thus,
\begin{align} \label{eq:PDEEstimate4_4a}
	c_1 \E \big( \| \bfu \|^2_{L^2(D)} \big) + \frac{c_0}{4} \E 
	\big( | \bfu |^2_{1,\pa D} \big) + \frac{1}{4} \E \big( | \bfu |^2_{1,D} \big)
	\leq c_2 \E \big( \| \bff \|^2_{L^2(D)} \big),
\end{align}
where
\begin{align*}
	c_1 &:= \frac{k^2}{4} - \frac{d-1}{2} \big( \gamma k^2 + \delta_4 \big) - \frac{1}
	{\mu K} \left( \frac{k^2 \gamma R \delta_1}{2} + \frac{R \delta_2}{2} 
	\right) \big(k^2(1 + \gamma) + \delta_5 + 2\mu k^2 \big) \\
	& \qquad - \left( \frac{k^2 R}{2} + \frac{k C_A R}{2 \delta_3}\right) 
	\frac{\delta_6}{c_A k} - \frac{k^2 \gamma R}{2 
	\delta_1}, \\
	c_2 &:= \frac{1}{2 \mu K \delta_5} \left( \frac{k^2 \gamma R \delta_1}{2}
	+ \frac{R \delta_2}{2} \right) + \frac{1}{4 \delta_6 c_A k} 
	\left( \frac{k^2 R}{2} + \frac{k C_A R}{2 \delta_3} \right) \\
	&\qquad \qquad + \frac{d-1}{8 
	\delta_4} + \frac{R}{2 \delta_2}.
\end{align*}
Setting
\begin{align*}
	\begin{array}{l l l}
	\delta_1 = \frac{1}{k}, & \delta_2 = \frac{\mu K}{16R(3 + 2 \mu)}, 
	& \delta_3 = \frac{1}{2kC_A} \min \{ K, 2 \mu K , c_0 \lambda, 2c_0 \mu \}, \\[.5cm]
	\delta_4 = \frac{1}{16(d-1)}, & \delta_5 = k^2, & \delta_6 = \frac{c_A 
	\delta_3 k^2}{16\big( kR \delta_3 + C_A R \big)},
	\end{array}
\end{align*}
and using the fact that $\gamma \leq 1$, which implies $\veps \leq \frac{1}{2}
$, yields
\begin{align*}
	c_1 = \frac{k^2}{4} - \frac{k^2}{2} \left( d-1 + \frac{3kR}{\mu K} + 
	\frac{2kR}{K} + kR \right) \gamma - \frac{3k^2}{32}.
\end{align*}
Also using the fact that $\gamma \leq \frac{1}{4} \left( d-1 + \frac{3kR}{\mu K} 
+ \frac{2kR}{K} + kR \right)^{-1}$ yields $c_1 \geq 
\frac{k^2}{8}$.  It is easy to check that $c_2 \leq C \left(1 + \frac{1}
{k^2} \right)$. Therefore, \eqref{eq:PDEEstimate4_4} becomes
\begin{align*}
	\frac{k^2}{8} \E \big( \| \bfu \|^2_{L^2(D)} \big) + \frac{c_0}{4} \E 
	\big( | \bfu |^2_{1,\pa D} \big) \leq C \left(1 + \frac{1}{k^2} \right) 
	\E \big( \| \bff \|^2_{L^2(D)} \big).
\end{align*}
Multiplying both sides by $\frac{8}{k^2}$ and applying 
\eqref{eq:PDEEstimate1_2} with $\delta_2 = k$ implies 
\eqref{eq:PDEEstimate4_1} with $\ta = 1$.  

\medskip
{\em Step 4:}
Now we prove \eqref{eq:PDEEstimate4_2} and \eqref{eq:PDEEstimate4_3}.
By \eqref{eq:Korn}, \eqref{eq:PDEEstimate1_1} with $\delta_1 = k^2$, and 
\eqref{eq:PDEEstimate4_1} the following holds.
\begin{align*}
	K \E \big( \| \bfu \|^2_{H^1(D)} \big) &\leq C \Big( \E \big( \| \bfu \|
	^2_{L^2(D)} \big) + \E \big( | \bfu |^2_{1,D} \big) \Big) \\
	& \leq C \left( \big(1 + (1+\veps)^2k^2 \big) \E \big( \| \bfu \|
	^2_{L^2(D)} \big) + \frac{1}{k^2} \E \big( \| \bff \|^2_{L^2(D)} \big) 
	\right) \\
	& \leq C \left( \big(1 + k^2 \big) \left(k^{\ta - 2} + \frac{1}{k^2} 
	\right)^2 + \frac{1}{k^2} \right) \E \big( \| \bff \|^2_{L^2(D)} \big) \\
	& \leq C \left(k^{\ta - 1} + \frac{1}{k^2} \right)^2 \E \big( \| \bff \|
	^2_{L^2(D)} \big).
\end{align*}
Here we have used the fact that $\veps \leq \frac{1}{2}$. Thus, 
\eqref{eq:PDEEstimate4_2} holds.

By \eqref{eq:PDEEstimate3_1} and \eqref{eq:PDEEstimate4_1} we find
\begin{align*}
	\E \big( \| \bfu \|^2_{H^2(D)} \big) & \leq C \big(1 + k^4 \big) \E \big( 
	\| \bfu \|^2_{L^2(D)} \big) + C \E \big( \| \bff \|^2_{L^2(D)} \big) \\
	& \leq C \big(1 + k^4 \big) \left(k^{\ta - 2} + \frac{1}{k^2} \right)^2 \E 
	\big( \| \bff \|^2_{L^2(D)} \big) \\
	& \leq C \left(k^{\ta} + \frac{1}{k^2} \right)^2 \E \big( \| \bff \|
	^2_{L^2(D)} \big).
\end{align*}
Thus \eqref{eq:PDEEstimate4_3} holds.
\end{proof}

\begin{remark} \label{rem:boundary_korn}
The Korn-type inequality on the boundary \eqref{eq:Stoch_Bound_Korn} was needed 
to obtain estimates that are optimal in the frequency $k$.  This is one key 
difference between the scalar Helmholtz problem and elastic Helmholtz problem.  
The parameter $\ta$ introduced in these solution estimates plays a key role in 
the analysis presented throughout the rest of the paper.
\end{remark}

\begin{theorem}
\label{thm:PDEExistence}
Let $\bff \in \bL^2( \Om, \bL^2(D))$.  For each fixed pair of positive 
numbers $k \geq 1$ and $\veps$ satisfying $\veps(2+\veps) < \gamma_0$, there 
exists a unique solution $\bfu \in \bL^2\big( \Om, \bV \big)$ 
to problem \eqref{Eq:WeakFormulation}--\eqref{Eq:WeakSesquilinearForm}.
\end{theorem}

\begin{proof}
The proof is based on the well-known Fredholm Alternative Principle (cf. 
\cite{Gilbarg_Trudinger_01}).  First, it is easy to check that the 
sesquilinear form in \eqref{Eq:WeakSesquilinearForm} satisfies a G\"{a}rding's 
inequality on $\bL^2\big(\Om,\bH^1(D)\big)$.  Second, to 
apply the Fredholm Alternative Principle we need to prove that the solution 
to the adjoint problem of 
\eqref{Eq:WeakFormulation}--\eqref{Eq:WeakSesquilinearForm} is unique.  It is 
easy to verify that the adjoint problem has an associated sesquilinear form 
\begin{align*}
	\widehat{a} (w,v) &:= 2\mu (\bsnabla\bfw,\bsnabla\bfv)_D + \lambda 
	(\ddiv \bfw, \ddiv \bfv)_D - k^2( \alpha^2 \bfw, \bfv)_D  
        - \i k \langle \alpha A \bfw, \bfv \rangle_{\pa D}.
\end{align*}
Note that $\widehat{a}(\cdot,\cdot)$ and $a(\cdot,\cdot)$ differ only in 
the sign of the last term.  As a result, all the solution estimates for 
problem \eqref{Eq:WeakFormulation}--\eqref{Eq:WeakSesquilinearForm} still 
hold for its adjoint problem.  Since the adjoint problem is linear, 
solution estimates immediately imply uniqueness.  Thus, the Fredholm 
Alternative Principle implies 
\eqref{Eq:WeakFormulation}--\eqref{Eq:WeakSesquilinearForm} has a unique 
solution $\bfu \in \bL^2\big( \Om,\bH^1(D) \big)$, which infers 
$\bfu \in \bL^2\big( \Om,\bV \big)$ . The proof is complete.
\end{proof}

\section{Multi-modes representation of the solution and its finite modes 
approximations} \label{sec:Multi-Modes}

Following the procedure set forth in \cite{Feng_Lin_Lorton_15}, the first goal of 
this section is to introduce and analyze a multi-modes representation for the 
solution to problem \eqref{Eq:WeakFormulation}--\eqref{Eq:WeakSesquilinearForm} 
in the form of a power series in terms of the parameter $\veps$.  
This multi-modes representation is key to obtaining an efficient numerical 
algorithm for estimating $\E (\bfu )$.  The second goal of this section is 
to estimate the error associated with approximating the solution $\bfu$ by a finite term 
truncation of its multi-modes representation.  We use the term finite 
modes approximation to refer to this truncation.

Due to the linear nature of the elastic Helmholtz operator 
as well as its similarities to the scalar Helmholtz operator, most of the 
results in this section are obtained in similar fashion as their respective 
counterparts in \cite{Feng_Lin_Lorton_15}, with changes due to the inherent 
difficulty associated with the elastic Helmholtz operator.  Let $\bfu^\veps$ 
denote the solution to problem 
\eqref{Eq:WeakFormulation}--\eqref{Eq:WeakSesquilinearForm} and assume that it 
can be represented in the form
\begin{align} \label{eq:MultiModes}
	\bfu^{\veps} = \sum^\infty_{n = 0} \veps^n \bfu_n.
\end{align}
The validity of this expansion will be proved later.  Without loss of 
generality, we assume that $k \geq 1$ and $D \subset B_1 (\mathbf{0})$.  
Otherwise, the problem can be re-scaled to this regime by a suitable change of 
variable.  We note that this normalization implies 
$R = 1$ in Theorem \ref{thm:PDEEstimate}.

Substituting the above expansion into the elastic Helmholtz equation
\eqref{Eq:ElasticPDE} and matching the coefficients of $\veps^n$ order terms
for $n=0,1,2,\cdots$, we obtain
\begin{align}\label{eq3.2}
\bfu_{-1} &:\equiv \mathbf{0},\\
-\ddiv(\sigma(\bfu_0))- k^2 \bfu_0 &= \bff, \label{eq3.3} \\
-\ddiv(\sigma(\bfu_n))- k^2 \bfu_{n} &= 2k^2\eta \bfu_{n-1} 
+k^2\eta^2\bfu_{n-2},
\qquad \mbox{for } n\geq 1. \label{eq3.4}
\end{align}
Since $\alpha|_{\pa D} = 1$ a.s., then substituting $\bfu^\veps$ 
into \eqref{Eq:ElasticBC} and matching coefficients of $\veps^n$ order terms
for $n=0,1,2,\cdots$, yields the same absorbing boundary condition for each mode 
function $\bfu_n$.  Namely,
\begin{equation}\label{eq3.5}
\sigma(\bfu_n) \bfnu + \i k A \bfu_n = \mathbf{0}, \qquad\mbox{for } n\geq 0.
\end{equation}

We observe that all the mode functions satisfy the same type of ``nearly deterministic"
elastic Helmholtz equations with the same boundary condition. 
The only difference in the equations are found in the right-hand side source terms.
In particular, the source term in \eqref{eq3.3} comes from the source term of 
\eqref{Eq:ElasticPDE} while the source term in \eqref{eq3.4} involves a three-term
recursive relationship involving the mode functions and the random field $\eta$.
This important feature will be utilized in Section \ref{sec:Numerical_Procedure} 
to construct our overall numerical methodology for solving problem 
\eqref{Eq:ElasticPDE}--\eqref{Eq:ElasticBC}.

Next, we address the existence and uniqueness of each mode function $\bfu_n$.

\begin{theorem}\label{thm3.1}
Let $\bff\in \bL^2(\Ome, \bL^2(D))$. Then for each $n\geq 0$, there
exists a unique solution $\bfu_n \in \bL^2(\Om, \bH^1(D))$ (understood in
the sense of Definition \ref{Def:WeakSolution}) to problem \eqref{eq3.3},
\eqref{eq3.5} for $n=0$
and problem \eqref{eq3.4},\eqref{eq3.5} for $n\geq 1$. Moreover, 
for $n\geq 0$, $\bfu_n$ satisfies
\begin{align}\label{eq3.6}
	\E \Big( \| \bfu_n \|^2_{L^2(D)} &+ \| \bfu_n \|_{L^2(\pa D)}^2 + 
	\frac{c_0}{k^2} | \bfu_n |^2_{1,\pa D} \Big) \\
	 &\leq \left(k^{\ta - 2} + \frac{1}{k^2} \right)^2 C(n,k) \E \big( \| \bff 
	\|^2_{L^2(D)} \big), \notag \\
	\E \big( \| \bfu_n \|^2_{H^1(D)} \big) & \leq \left( k^{\ta - 1} + \frac{1}
	{k^2} \right)^2 C(n,k) \E \big( \| \bff \|^2_{L^2(D)} \big), 
	\label{eq3.7}
\end{align}
where
\begin{equation}\label{eq3.7a}
C(0,k):=C_0,\qquad C(n,k):= 4^{2n-1}C_0^{n+1} (1+k^{\ta})^{2n} \quad\mbox{for } 
n\geq 1,
\end{equation}
and $\ta$ is defined in \eqref{eq:ta}.
Moreover, if $\bfu_n \in \bL^2(\Ome, \bH^2(D))$, there also holds
\begin{align}\label{eq3.6a}
\E \big( \| \bfu_n \|^2_{H^2(D)} \big) \leq  \left(k^{\ta} + \frac{1}
	{k^2} \right)^2 C(n,k) \E \big( \| \bff \|^2_{L^2(D)} \big). 
\end{align}

\end{theorem}

\begin{proof}
The proof for this theorem mimics that of Theorem 3.1 from 
\cite{Feng_Lin_Lorton_15} with only minor changes necessary.  In particular, the 
dependencies on $k$ in the solution estimates present in Theorem 
\ref{thm:PDEEstimate} are different than their respective counterparts in 
\cite{Feng_Lin_Lorton_15}. This in-turn changes the form of $C(n,k)$.  Also, in 
this paper we take $\eta \in \bL^2(\Om,\bL^\infty(D))$ ensuring that 
$\alpha|_{\pa D} = 1$ a.s.

For each $n\geq 0$, the PDE problem associated with $\bfu_n$ is the same type
of elastic Helmholtz problem as the original problem 
\eqref{Eq:ElasticPDE}--\eqref{Eq:ElasticBC} (with $\veps=0$ in the left-hand 
side of the PDE). Hence, all a priori estimates of Theorem \ref{thm:PDEEstimate} 
hold for each $\bfu_n$ (with its respective right-hand source function).  
First, we have
\begin{align}\label{eq3.8}
&\E\Bigl( \norm{\bfu_0}_{L^2(D)}^2 +\norm{\bfu_0}_{L^2(\p D)}^2 
+ \frac{c_0}{k^2} | \bfu_0 |_{1,\pa D}^2 \Bigr) 
\\
&\hspace*{2.1cm}\leq C_0 \Bigl( k^{\ta - 2} + \frac{1}{k^2} \Bigr)^2 \E(\norm{\bff}_{L^2(D)} ^2), \\ 
&\E(\norm{\bfu_0}_{H^1(D)}^2) \leq C_0 \Bigl(k^{\ta-1} +\frac{1}{k^2} \Bigr)^2
\E(\norm{\bff}_{L^2(D)}^2), \label{eq3.9} \\
&\E \big( \| \bfu_0 \|^2_{H^2(D)} \big) \leq C_0 \left( k^{\ta} + \frac{1}
{k^2} \right)^2 \E \big( \| \bff \|^2_{L^2(D)} \big). \label{eq3.9a}
\end{align}
Thus, \eqref{eq3.6}, \eqref{eq3.7}, and \eqref{eq3.6a} hold for $n=0$.

Next, we use induction to prove that \eqref{eq3.6}, \eqref{eq3.7}, \eqref{eq3.6a} 
for $n> 0$.  Assume that \eqref{eq3.6}, \eqref{eq3.7}, and \eqref{eq3.6a} hold for 
all $0\leq n\leq \ell-1$, then
\begin{align*}
\E\Bigl( &\norm{\bfu_\ell}_{L^2(D)}^2 + \norm{\bfu_\ell}_{L^2(\p D)}^2 
+ \frac{c_0}{k^2}| \bfu_\ell |_{1,\pa D}^2 \Bigr) \\
&\,
\leq 2C_0\Bigl( k^{\ta - 2} +\frac{1}{k^2} \Bigr)^2
\E\Bigl( \norm{2k^2\eta \bfu_{\ell-1} }_{L^2(D)}^2
+\overline{\delta}_{1\ell} \norm{k^2\eta^2 \bfu_{\ell-2}}_{L^2(D)}^2 \Bigr) \\
&\,
\leq 2C_0\Bigl( k^{\ta - 2} +\frac{1}{k^2} \Bigr)^2
\big(1+k^{\ta}\big)^2 \Bigl( 4C(\ell-1,k) +C(\ell-2,k) \Bigr) \E(\norm{\bff}
_{L^2(D)}^2)\\
&\,
\leq \Bigl( k^{\ta - 2} +\frac{1}{k^2} \Bigr)^2 \, 8 C_0 \big(1+k^{\ta}\big)^2 
C(\ell-1,k) \left( 1+ \frac{C(\ell-2,k)}{C(\ell-1,k)} \right) \E(\norm{\bff}
_{L^2(D)}^2)\\
&\,
\leq \Bigl( k^{\ta - 2}+\frac{1}{k^2} \Bigr)^2 C(\ell, k) \E(\norm{\bff}
_{L^2(D)}^2),
\end{align*}
where $\overline{\delta}_{1\ell}=1-\delta_{1\ell}$ and $\delta_{1\ell}$ denotes
the Kronecker delta, and we have used the fact that $k \geq 1$ and
\[
8C_0 \big(1+k^{\ta}\big)^2 C(\ell-1,k) \left( 1+ \frac{C(\ell-2,k)}{C(\ell-1,k)} 
\right) \leq C(\ell, k).
\]
Similarly, we have
\begin{align*}
&\E\bigl(\norm{\bfu_\ell}_{H^1(D)}^2 \bigr)
\leq 2C_0\Bigl( k^{\ta - 1} + \frac{1}{k^2} \Bigr)^2
\E\Bigl( \norm{2k^2\eta \bfu_{\ell-1} }_{L^2(D)}^2
+\overline{\delta}_{1\ell} \norm{k^2\eta^2 \bfu_{\ell-2}}_{L^2(D)}^2 \Bigr) \\
& \qquad \leq 2C_0\Bigl(k^{\ta - 1} +\frac{1}{k^2} \Bigr)^2
\big(1+k^{\ta}\big)^2 \Bigl( 4C(\ell-1,k)+C(\ell-2,k) \Bigr) \E\bigl(\norm{\bff}
_{L^2(D)}^2\bigr)\\
& \qquad \leq \Bigl( k^{\ta - 1} +\frac{1}{k^2} \Bigr)^2 C(\ell, k) \E\bigl(\norm{\bff}
_{L^2(D)}^2\bigr),\\
&\E\bigl(\norm{\bfu_n}_{H^2(D)}^2 \bigr) \leq 2C_0\Bigl(k^{\ta} +\frac{1}{k^2} 
\Bigr)^2 \E\Bigl( \norm{2k^2\eta \bfu_{n-1} }_{L^2(D)}^2
+\overline{\delta}_{1\ell} \norm{k^2\eta^2 \bfu_{n-2}}_{L^2(D)}^2 \Bigr) \\
& \qquad \leq 2 C_0\Bigl(k^{\ta} +\frac{1}{k^2} \Bigr)^2
\big(1+k^{\ta} \big)^2 \Bigl( 4C(n-1,k)+C(n-2,k) \Bigr) \E\bigl(\norm{\bff}
_{L^2(D)}^2\bigr)\\
& \qquad \leq  \Bigl(k^{\ta}+\frac{1}{k^2} \Bigr)^2 C(n, k) 
\E\bigl(\norm{\bff}_{L^2(D)}^2\bigr).
\end{align*}
Hence, \eqref{eq3.6}, \eqref{eq3.7}, and \eqref{eq3.6a} hold for $n=\ell$. Thus, 
the induction is complete.


With a priori estimates \eqref{eq3.6}, \eqref{eq3.7}, \eqref{eq3.6a} in hand,
the proof of existence and uniqueness for each $\bfu_n$ follows verbatim the 
proof of Theorem \ref{thm:PDEExistence}, which we leave to the interested reader 
to verify. The proof is complete.
\end{proof}

Now we are ready to justify the multi-modes representation \eqref{eq:MultiModes} 
for the solution $\bfu^\veps$ of problem 
\eqref{Eq:WeakFormulation}--\eqref{Eq:WeakSesquilinearForm}.

\begin{theorem}\label{thm3.2}
Let $\{\bfu_n\}$ be the same as in Theorem \ref{thm3.1}. Then
\eqref{eq:MultiModes} is valid in $\bL^2(\Ome, \bH^1(D))$ provided that
$c_\veps:=4\veps C_0^{\frac12}\big(1+k^{\ta}\big)<1$.
\end{theorem}

\begin{proof}
The proof consists of two parts: (i) to show the infinite
series on the right-hand side of \eqref{eq:MultiModes} converges in $\bL^2(\Ome, 
\bH^1(D))$; (ii) to show the limit coincides with the solution $\bfu^\veps$.
To prove (i), we define the partial sum
\begin{equation}\label{eq3.12}
\bfU^\veps_N:= \sum_{n=0}^{N-1} \veps^n \bfu_n.
\end{equation}
Then for any fixed positive integer $p$ we have
\[
\bfU^\veps_{N+p} -\bfU^\veps_N= \sum_{n=N}^{N+p-1} \veps^n \bfu_n.
\]
It follows from Schwarz inequality and \eqref{eq3.6} that for $j=0,1$
\begin{align*}
&\E\bigl( \|\bfU^\veps_{N+p} -\bfU^\veps_N\|_{H^j(D)}^2 \bigr)
\leq p \sum_{n=N}^{N+p-1} \veps^{2n} \E(\|\bfu_n\|_{H^j(D)}^2) \\
&\hskip 0.5in
\leq p \Bigl(k^{\ta + j - 2} +\frac{1}{k^2} \Bigr)^2 \E(\norm{\bff}_{L^2(D)}^2)
\sum_{n=N}^{N+p-1} \veps^{2n} C(n,k) \nonumber\\
&\hskip 0.5in
\leq C_0p\Bigl(k^{\ta + j - 2} +\frac{1}{k^2} \Bigr)^2 \E(\norm{\bff}_{L^2(D)}
^2)\sum_{n=N}^{N+p-1} c_\veps^{2n}  \nonumber\\
&\hskip 0.5in
\leq C_0p \Bigl(k^{\ta + j - 2} +\frac{1}{k^2} \Bigr)^2 \E(\norm{\bff}_{L^2(D)}
^2)\cdot \frac{c_\veps^{2N} \bigl(1-c_\veps^{2p}\bigr)}{1-c_\veps^2}.  \nonumber
\end{align*}
Thus, if $c_\veps<1$ we have
\[
\lim_{N\to \infty} \E\bigl( \|\bfU^\veps_{N+p} -\bfU^\veps_N\|_{H^1(D)}^2 
\bigr)=0.
\]
Therefore, $\{\bfU^\veps_N\}$ is a Cauchy sequence in $\bL^2(\Ome, \bH^1(D))$.
Since $\bL^2(\Ome, \bH^1(D))$ is a Banach space, then there exists a function
$\bfU^\veps\in \bL^2(\Ome, \bH^1(D))$ such that
\[
\lim_{N\to \infty} \bfU^\veps_N =\bfU^\veps \qquad\mbox{in } \bL^2(\Ome,
\bH^1(D)).
\]

To show (ii),  note that by the definitions of $\bfu_n$ and $\bfU^
\veps_N$,
it is easy to check that $\bfU^\veps_N$ satisfies
\begin{align}\label{eq3.13}
&\int_\Ome a\big(\bfU^\veps_N, \bfv\big) \, dP  \\
& \qquad \qquad  = \int_\Ome (\bff, \bfv)_D \, dP - k^2 \veps^N \int_\Ome \bigl( 
\eta (2+\veps\eta) \bfu_{N-1}
+ \eta^2 \bfu_{N-2},\, \bfv \bigr)_D\, dP \notag
\end{align}
for all $\bfv\in \bL^2(\Ome, \bH^1(D))$.
In other words, $\bfU_N^\veps$ solves the following elastic Helmholtz problem:
\begin{alignat*}{2}
-\ddiv \Big( \sigma \big( \bfU^\veps_N \big) \Big) - k^2 \alpha^2 \bfU^\veps_N
&=\bff- k^2 \veps^N \bigl( \eta (2+\veps\eta) \bfu_{N-1} + \eta^2 \bfu_{N-2} 
\bigr)
&&\qquad\mbox{in } D,\\
\sigma \big(\bfU^\veps_N \big) \bfnu + \i k A \bfU^\veps_N &= \mathbf{0} &&\qquad\mbox{on } \p D,
\end{alignat*}
in the sense of Definition \ref{Def:WeakSolution}.

By \eqref{eq3.6} and Schwarz inequality we have
\begin{align*}
&k^2 \veps^N \left| \int_\Ome \bigl( \eta (2+\veps\eta) \bfu_{N-1}
+ \eta^2 \bfu_{N-2},\, v \bigr)_D\, dP \right|\\
&\qquad
\leq 3k^2\veps^N \Bigl( \bigl(\E(\|\bfu_{N-1}\|_{L^2(D)}^2)\bigr)^{\frac12}
+ \bigl(\E(\|\bfu_{N-2}\|_{L^2(D)}^2)\bigr)^{\frac12} \Bigr)
\bigl(\E(\|\bfv\|_{L^2(D)}^2)\bigr)^{\frac12} \nonumber\\
&\qquad
\leq 6k^2\veps^N \Bigl(k^{\ta - 2} +\frac{1}{k^2} \Bigr)
C(N-1,k)^{\frac12} \bigl(\E(\|\bff\|_{L^2(D)}^2)\bigr)^{\frac12}
\bigl(\E(\|\bfv\|_{L^2(D)}^2)\bigr)^{\frac12} \nonumber \\
&\qquad
\leq 3\veps (k^{\ta}+1)C_0^{\frac12} c_\veps^{N-1} \bigl(\E(\|\bff\|_{L^2(D)}^2)
\bigr)^{\frac12}
\bigl(\E(\|\bfv\|_{L^2(D)}^2)\bigr)^{\frac12} \nonumber \\
&\qquad
\longrightarrow 0 \quad\mbox{as } N\to \infty\quad\mbox{provided that } c_
\veps<1.
\nonumber
\end{align*}

Setting $N\to \infty$ in \eqref{eq3.13} immediately yields
\begin{align}\label{eq3.15}
&\int_\Ome a\big(\bfU^\veps, \bfv\big) \, dP   = \int_\Ome (\bff, \bfv)_D \, 
dP.
\end{align}
Thus, $\bfU^\veps$ is a solution to problem 
\eqref{Eq:ElasticPDE}--\eqref{Eq:ElasticBC}, in the sense of Definition
\ref{Def:WeakSolution}. By the uniqueness of the solution, we conclude that $
\bfU^\veps=\bfu^\veps$. Therefore, \eqref{eq:MultiModes} holds in $\bL^2(\Ome,
\bH^1(D))$. The proof is complete.
\end{proof}

The above proof also implies an upper bound for the error $\bfu^\veps- \bfU^
\veps_N$ as stated in the next theorem.

\begin{theorem}\label{thm3.3}
Let $\bfU^\veps_N$ be the same as above and $\bfu^{\veps}$ denote the solution 
to problem \eqref{Eq:ElasticPDE}--\eqref{Eq:ElasticBC} (in the sense of 
Definition \ref{Def:WeakSolution}), and $c_\veps:=4\veps C_0^{\frac12}
\big(1+k^{\ta}\big)$. Then there holds for $\veps(2\veps+1)<\gamma_0$
\begin{align}\label{eq3.16}
&\E(\norm{\bfu^\veps - \bfU^\veps_N}_{H^j(D)}^2)\\
&\qquad \leq \frac{9 C_0c_\veps^{2N}}{32\big(1+k^{\ta}\big)^2} \Bigl(k^{\ta + j - 1} +
\frac{1}{k} \Bigr)^4 \E(\|\bff\|_{L^2(D)}^2), \quad j=0,1, \notag
\end{align}
provided that $c_\veps<1$. Where $C$ is a positive constant independent of $k$ 
and $\veps$.
\end{theorem}

\begin{proof}
Let $\bfpsi^{\veps}_N:= \bfu^\veps - \bfU^\veps_N$, subtracting \eqref{eq3.13} 
from \eqref{eq3.15} we get
\begin{align}\label{eq3.17}
&\int_\Ome a\big( \bfpsi^{\veps}_N, \bfv \big)_D \, dP = k^2 \veps^N \int_\Ome 
\bigl( \eta (2+\veps\eta) \bfu_{N-1} + \eta^2 \bfu_{N-2},\, \bfv \bigr)_D\, dP. 
\end{align}
In other words, $\bfpsi^\veps_N$ solves the following Helmholtz problem:
\begin{alignat*}{2}
-\ddiv \Big( \sigma\big( \bfpsi^\veps_N\big) \Big) - k^2 \alpha^2 \bfpsi^\veps_N
&= k^2 \veps^N \bigl( \eta (2+\veps\eta) \bfu_{N-1} + \eta^2 \bfu_{N-2} \bigr)
&&\qquad\mbox{in } D,\\
\p_\nu \bfpsi^\veps_N + \i k\alpha \bfpsi^\veps_N &= \mathbf{0} &&\qquad\mbox{on 
} \p D,
\end{alignat*}
in the sense of Definition \ref{Def:WeakSolution}.

By Theorem \ref{thm:PDEEstimate} and \eqref{eq3.6} we obtain for $j=0,1$ there holds
\begin{align*}
&\E(\|\bfpsi^\veps_N\|_{H^j(D)}^2) \\
& \qquad \leq 18 C_0 \Bigl(k^{\ta + j - 2} +\frac{1}{k^2} \Bigr)^2\,
\Bigl[k^4\veps^{2N} \Bigl(\E(\|\bfu_{N-1}\|_{L^2(D)}^2) + \E(\|\bfu_{N-2}\|
_{L^2(D)}^2) \Bigr) \Bigr] \\
& \qquad \leq 18C_0 k^4\veps^{2N} \Bigl(k^{\ta + j - 2} +\frac{1}{k^2} \Bigr)^4
C(N-1,k)\, \E(\|\bff\|_{L^2(D)}^2) \nonumber  \\
& \qquad \leq \frac{18C_0 c_{\veps}^{2N}}{64\big(1+k^{\ta}\big)^2} \Bigl(k^{\ta + j - 1} +
\frac{1}{k} \Bigr)^4 \E(\|\bff\|_{L^2(D)}^2). \nonumber
\end{align*}
The proof is complete.
\end{proof}

\begin{remark} \label{rem:eps_size}
	The condition $c_\veps < 1$, which is used to ensure that the multi-modes expansion 
	\eqref{eq:MultiModes} is valid, is of the form $\veps = O\big( k^{-\ta} 
	\big)$.  In the case that \eqref{eq:Stoch_Bound_Korn} holds, this 
	restriction on the size of the perturbation parameter $\veps$ takes the form 
	$\veps = O \big( k^{-1} \big)$.  This matches the analogous result for the 
	scalar Helmholtz problem in weakly random media in 
	\cite{Feng_Lin_Lorton_15}.
\end{remark}

\section{Monte Carlo discontinuous Galerkin approximation of the truncated 
multi-modes expansion $\bfU^{\veps}_N$} \label{sec:MCIP-DG}

In the previous sections, we prove a multi-modes representation for the solution to
\eqref{Eq:ElasticPDE}--\eqref{Eq:ElasticBC} and also derive the convergence rate for 
its finite modes approximation $\bfU^\veps_N$. 
Our overall numerical methodology is based on approximating $
\bfu^\veps$ by its finite modes expansion $\bfU^\veps_N$.  Thus, to approximate 
$\E (\bfu^\veps)$ we need to compute the expectations $\{ \E (\bfu_n) \}$ of the 
first $N$ mode functions $\{ \bfu_n \}_{n = 0}^{N-1}$.  This requires the use of 
an accurate and robust numerical (discretization) method for the ``nearly 
deterministic" elastic Helmholtz problems \eqref{eq3.3}, \eqref{eq3.5} and 
\eqref{eq3.4}, \eqref{eq3.5}. The construction of such a numerical method is the 
focus of this section.  Clearly, $\E(\bfu_n)$ cannot be computed 
directly for $n \geq 1$ due to the multiplicative nature of the right-hand side 
of \eqref{eq3.4}.  On the other hand, $\E(\bfu_0)$ can be computed directly 
because it satisfies a deterministic elastic Helmholtz equation with right-hand 
side $\E(\bff)$ and a homogeneous boundary condition.

The goal of this section is to develop a Monte Carlo interior penalty 
discontinuous Galerkin (MCIP-DG) method for the above mentioned elastic 
Helmholtz problems.  Our MCIP-DG method is a direct generalization of the 
deterministic IP-DG method proposed by us in \cite{Feng_Lorton_15} for the related 
deterministic elastic Helmholtz problem.  This IP-DG method was chosen because 
it is shown to perform well in the case of a large frequency $k$.  In 
particular, this IP-DG method is shown to be unconditionally stable (i.e. 
stable without a mesh constraint) and optimally convergent in the mesh parameter 
$h$.

\subsection{DG notation}
To introduce the IP-DG method we need to start with some standard notation 
used in the DG community. Let $\cT_h$ be a quasi-uniform partition of $D$ such 
that $\overline{D}=\bigcup_{K\in\cT_h} \overline{K}$. Let $h_K$ denote
the diameter of $K\in \cT_h$ and $h:=\mbox{max}\{h_K; K\in\cT_h\}$.
$\bH^s(\cT_h)$ denotes the standard broken Sobolev space and $\bV^h$ denotes
the DG finite element space which are both defined as
\[
\bH^s(\cT_h):=\prod_{K\in\cT_h} \bH^{s}(K), \qquad \bV^h:=\prod_{K\in\cT_h} 
\left( P_1(K) \right)^d,
\]
where $P_1(K)$ is the set of all polynomials of degree less than or equal to 1.  
Let $\cE_h^I$ denote the set of all interior faces/edges of $\cT_h$, $\cE_h^B$ 
denote the set of all boundary faces/edges of $\cT_h$, and $\cE_h:=\cE_h^I\cup \cE_h^B
$. We define the following two $L^2$-inner products for piecewise continuous functions over 
the mesh $\cT_h$ 
\[
(\bfv,\bfw)_{\cT_h}:= \sum_{K\in \cT_h} \int_{K} \bfv \cdot \overline{\bfw} \, dx, 
\qquad \Langle \bfv,\bfw\Rangle_{\cS_h} :=\sum_{e\in \cS_h} \int_e \bfv \cdot 
\overline{\bfw} \, dS,
\]
for any set $\cS_h \subset \cE_h$. 

Let $K, K'\in \cT_h$ and $e=\partial K\cap \partial K'$ and assume
global labeling number of $K$ is smaller than that of $K'$.
We choose $\bfnu_e:= \bfnu_K|_e=-\bfnu_{K'}|_e$ as the unit normal on $e$ 
outward to $K$ and define the following standard jump and average notations 
across the face/edge $e$:
\begin{alignat*}{4}
[\bfv] &:= \bfv|_K- \bfv|_{K'}
\quad &&\mbox{on } e\in \cE_h^I,\qquad
&&[\bfv] := \bfv\quad
&&\mbox{on } e\in \cE_h^B,\\
\{ \bfv\} &:=\frac12\bigl( \bfv|_K + \bfv|_{K'} \bigr) \quad
&&\mbox{on } e\in \cE_h^I,\qquad
&&\{\bfv\}:= \bfv\quad
&&\mbox{on } e\in \cE_h^B,
\end{alignat*}
for $\bfv\in \bV^h$. We also define the following semi-norms on $\bH^s(\cT_h)$:
\begin{align*}
| \bfv |_{1,h} &:= \left ( \sum_{K \in \mathcal{T}_h} \lambda \| \ddiv \bfv \|
_{L^2(K)}^2 + 2\mu \| \bsnabla\bfv \|_{L^2(K)}^2 \right)^{ \frac{1}{2}}, \\
\|\bfv \|_{1,h} &:=  \left( |\bfv |^2_{1,h} 
+ \sum_{e \in \mathcal{E}_h^I} \Big( \frac{\gamma_{0,e}}{h_e} \big \| [\bfv ] 
\big \|^2_{L^2(e)} + \gamma_{1,e} h_e \big \| [\sigma(\bfv) \bfn_e] \big \|
^2_{L^2(e)} \Big) \right)^{\frac{1}{2}}, 
\end{align*}
\begin{align*}
|||\bfv|||_{1,h} &:= \left( \| \bfv \|^2_{1,h} + \sum_{e \in \mathcal{E}_h^I}
\frac{h_e}{\gamma_{0,e}}  \big \| \{\sigma(\bfv) \bfn_e\} \big \|^2_{L^2(e)} 
\right)^{\frac{1}{2}}.
\end{align*}
Here $\{\gamma_{0,e}\}$ and
$\{\gamma_{1,e}\}$ are penalty parameters to be discussed in more detail in the next subsection.

\subsection{IP-DG method for the deterministic elastic Helmholtz problem}

In this subsection we consider the following deterministic elastic Helmholtz 
problem and its IP-DG approximation proposed in \cite{Feng_Lorton_15}
\begin{align}
	- \bdiv(\sigma(\bfPhi)) - k^2 \bfPhi &= \bfF \qquad \mbox{ in } D, 
	\label{Eq:DeterministicPDE} \\
	\sigma(\bfPhi) \bfnu + \bfi k A \bfPhi &= \mathbf{0} \qquad \mbox{ on } \pa 
	D. 
\label{Eq:DeterministicBC}
\end{align}
We note that when $\bfF = \E (\bff )$, the solution to 
\eqref{Eq:DeterministicPDE}--\eqref{Eq:DeterministicBC} is $\bfPhi = \E (\bfu_0)
$.  Thus, all the results of this subsection apply to the mean of the first mode 
function $\bfu_0$.

The IP-DG weak formulation of 
\eqref{Eq:DeterministicPDE}--\eqref{Eq:DeterministicBC} is defined as (cf. 
\cite{Feng_Lorton_15}) seeking $\bfPhi \in \bH^1(D)$ such that 
\begin{align}
	a_h( \bfPhi, \bfpsi ) = (\bfF, \bfpsi)_D \qquad \forall \bfpsi \in \bH^1(D), 
	\label{Eq:IPDG_Weak_Form_1}
\end{align}
where
\begin{align} \label{Eq:IPDF_Weak_Form_2}
	a_h(\bfphi, \bfpsi) &:= b_h(\bfphi, \bfpsi) - k^2 (\bfphi,\bfpsi)_D + \i k 
	\langle A \bfphi, \bfpsi \rangle_{\pa D} \\
	& \qquad + \i \big( J_0 (\bfphi, \bfpsi) + 
	J_1(\bfphi, \bfpsi) \big), \notag \\[.3cm] 
	b_h(\bfphi, \bfpsi) &:=  \lambda \big(\ddiv \bfphi, \ddiv \bfpsi 
	\big)_{\mathcal{T}_h} + 2 \mu \big(\bsnabla\bfphi, \bsnabla\bfpsi 
	\big)_{\mathcal{T}_h} - \bigl \langle \{ \sigma(\bfphi) \bfn_e \}, [\bfpsi] 
	\big\rangle_{\mathcal{E}^I_h} \notag \\
	& \qquad - \big \langle [\bfphi],  \{\sigma(\bfpsi) \bfn_e \}\big 
	\rangle_{\mathcal{E}^I_h}, \notag \\[.3cm]
	J_0 (\bfphi, \bfpsi) &:= \sum_{e \in \mathcal{E}^I_h} \frac{\gamma_{0,e}}
	{h_e} \big \langle[\bfphi], [\bfpsi] \big \rangle, \notag \\
	J_1 (\bfphi, \bfpsi) &:= \sum_{e \in \mathcal{E}^I_h} \gamma_{1,e} h_e \big
	\langle [\sigma(\bfphi) \bfn_e], [\sigma(\bfpsi) \bfn_e] \big\rangle . 
	\notag
\end{align}

\begin{remark} \label{Rem:Penalty_Terms}
	$J_0$ and $J_1$ are called interior penalty terms and the constants $
	\gamma_{0,e}$ and $\gamma_{1,e}$ are called penalty parameters and are taken 
	to be real-valued constants for each edge/face $e \in \mathcal{E}^I_h$.  These 
	terms are necessary components of a convergent IP-DG method and are used to 
	enforce continuity along element edges/faces and
	enhance the coercivity of the sesquilinear form $a_h(\cdot,\cdot)$.  Here we 
	note that $J_0$ is used to penalize jumps of function values over element edges
        and $J_1$ penalizes jumps of the normal stress $\sigma(\bfphi) 
	\bfn_e$ over element edges/faces.  Another interesting feature of this IP-DG 
	formulation is the use of purely imaginary penalization evident in the 
	multiplication of the imaginary unit $\i$ to the penalty terms $J_0$ and 
	$J_1$ in \eqref{Eq:IPDF_Weak_Form_2}.  Purely imaginary penalization is a 
	key component to ensuring that the IP-DG method for the deterministic 
	elastic Helmholtz equations is unconditionally stable.
	Purely imaginary penalty parameters also yield unconditional stability for the 
	acoustic Helmholtz equation and the time-harmonic Maxwell's equations \cite{Feng_Wu_09,Feng_Wu_14}.
\end{remark}

Following \cite{Feng_Lorton_15}, our IP-DG method for the deterministic elastic 
Helmholtz problem \eqref{Eq:DeterministicPDE}--\eqref{Eq:DeterministicBC} is 
defined as seeking $\bfPhi^h \in \bV^h$ such that
\begin{align}
	a_h( \bfPhi^h, \bfpsi^h ) = (\bfF, \bfpsi^h)_D \qquad \forall \bfpsi^h \in \bV^h. \label{Eq:IPDG_Weak_Form_3}
\end{align}
In \cite{Feng_Lorton_15,Lorton_14} it was proved that the above IP-DG method is 
unconditionally stable and its solutions satisfy some frequency-explicit 
stability estimates.  Its solutions also satisfy optimal order (in $h$) error 
estimates.  These results are summarized below in the following theorems.  
To make the constants in these theorems more tractable we assume that
$\gamma_{0,e}\geq \gamma_0>0$ and $\gamma_{1,e}\geq \gamma_1>0$. We also use $C$ 
to denote a generic positive constant independent of all other parameters in this paper.

\begin{theorem} \label{thm:IPDG_Stability}
	Let $\bfPhi^h \in \bV^h$ solve \eqref{Eq:IPDG_Weak_Form_3} for some $\bfF \in \bL^2(D)$. 
	
	(i) Then for any $k > 0$ and $h > 0$ there exists a positive constant  $
	\hat{C}_0$ independent of $k$, $h$, $\gamma_{0}$, and $\gamma_{1}$ such that 
	the following stability estimate holds
	\begin{align}\label{Eq:IPDG_Stability_1}
		&\| \bfPhi^h \|_{L^2(D)} +  \| \bfPhi^h \|_{L^2(\pa D)} + \frac{1}{k} \| 
		\bfPhi^h \|_{1,h} \\
		& \hspace*{2cm} \leq \frac{\hat{C}_0}{k} \left(C_s^2 + C_s + 
		\frac{C_s}{k} \right)^{\frac{1}{2}} \| \bfF \|_{L^2(D)}, \notag
	\end{align}
	where
	\begin{align*}
		C_s :=  \frac{\gamma_{0} + 1}{k \gamma_{0}} + \frac{1}{k^2 h} + \frac{1}
		{k^3 h^2 \gamma_{1}}.
	\end{align*}
	
	(ii) If $k \geq 1$ and $k^{\ta + 1} h = O(1)$, then there exists a positive 
	constant $\hat{C}_0$ independent of $k$ and $h$ such that
	\begin{align} \label{Eq:IPDG_Stability_2}
		\| \bfPhi^h \|_{L^2(D)} + \frac{1}{k}\| \bfPhi^h \|_{1,h}  \leq \hat{C}
		_0 \left(k^{\ta - 2} + \frac{1}{k^{2}} \right) \| \bfF \|_{L^2(D)}.
	\end{align}
\end{theorem}

\begin{remark} \label{rem:Asymptotic_Preasymptotic_Regimes}
	The condition $k^{\ta + 1} h = O(1)$ is a constraint on the mesh size for 
	fixed frequency $k$.  When $h$ is chosen to satisfy this constraint 
	the approximation method is said to be in the asymptotic mesh regime.  One 
	advantage of the above IP-DG method is that stability is ensured even 
	in the pre-asymptotic mesh regime (i.e. when $h$ does not satisfy $k^{\ta + 
	1} h = O(1)$).
\end{remark}

As an immediate consequence of Theorem \ref{thm:IPDG_Stability} we obtain the 
following unconditional solvability and uniqueness result for 
\eqref{Eq:IPDG_Weak_Form_3}.
\begin{corollary}
	For every $k,h >0$, and $\bfF \in \bL^2(D)$ there exists a unique solution $
	\bfPhi^h \in \bV^h$ to \eqref{Eq:IPDG_Weak_Form_3}.
\end{corollary}

\begin{theorem} \label{thm:IPDG_Error}
	Let $\bfPhi \in \bH^2(D)$ solve 
	\eqref{Eq:DeterministicPDE}--\eqref{Eq:DeterministicBC} and $\bfPhi^h \in \bV^h$ solve 
	\eqref{Eq:IPDG_Weak_Form_3} for $k,h > 0$. 
Suppose $\gamma_{0,e}\geq \gamma_0>0, \gamma_{1,e}\geq \gamma_1>0$ and let $\xi := 1 + \gamma_0^{-1}$.
	
	(i) For all $k,h > 0$, there exists a positive constant $C$ 
	independent of $k$, $h$, $\gamma_0$, and $\gamma_1$ such that
	\begin{align} \label{Eq:IPDG_Error_1}
		&\| \bfPhi - \bfPhi^h \|_{1,h} + k \| \bfPhi - \bfPhi^h \|_{L^2(D)}  \\
		& \hskip 0.8in  \leq C \xi^2 \Big(h + k h^2 \big(1 + k C_s \big) \Big) \big( 
		\xi + \gamma_1 + k h \big) \| \bfPhi \|_{H^2(D)}, \notag \\
		&\| \bfPhi - \bfPhi^h \|_{L^2(D)} \label{Eq:IPDG_Error_2} 
                 \leq C \xi^2 h^2 \big(1 + k C_s \big)\big( \xi + \gamma_1 + k h 
		\big) \| \bfPhi \|_{H^2(D)},
	\end{align}
	where $C_s$ is the positive constant defined in Theorem 
	\ref{thm:IPDG_Stability} part (i).
	
	(ii) If $k \geq 1$ and $k^{\ta + 1} h = O(1)$, then there exists a positive 
	constant $C$ independent of $k$, $h$, $\gamma_0$, and $\gamma_1$ 
	such that
	\begin{align}
	\| \bfPhi - \bfPhi^h \|_{1,h} &\leq C \xi^{4} (\xi + \gamma_1)\left(h + k 
	h^2 \right) \| \bfPhi \|_{H^2(D)}, \label{Eq:IPDG_Error_3} \\ 
	\| \bfPhi - \bfPhi^h \|_{L^2(D)} &\leq C \xi^{6} (\xi + \gamma_1)^2 
	\left(h^2 + k h^3 \right) \| \bfPhi \|_{H^2(D)}. \label{Eq:IPDG_Error_4}
\end{align}
\end{theorem}

\begin{remark} \label{rem:IPDG_Error}
Here the error is shown to be optimal in the mesh size $h$ in the asymptotic 
mesh regime.  On the other hand, in the pre-asymptotic mesh regime the error 
is only sub-optimal in $h$ because $C_s$ in \eqref{Eq:IPDG_Error_1} and 
\eqref{Eq:IPDG_Error_2} depends in an adverse way on $h$.
\end{remark}

For the rest of the paper we restrict our focus to the asymptotic mesh 
regime.  In other words we choose $h$ small enough to satisfy the 
condition $k^{\ta + 1} h = O(1)$.  This choice is made only to simplify the 
analysis later in the paper by allowing us to make use of 
\eqref{Eq:IPDG_Stability_2}, \eqref{Eq:IPDG_Error_3}, and 
\eqref{Eq:IPDG_Error_4}.

\subsection{MCIP-DG method for approximating $\mathbf{\E(\bfu_n)}$ for $
\mathbf{n \geq 0}$}

Recall that the mode function $\bfu_n$ solves the ``nearly deterministic" 
elastic Helmholtz problem
\begin{align}\label{Eq:Nearly_Deterministic_1}
-\ddiv(\sigma(\bfu_n)) - k^2 \bfu_{n} &= \bfS_n \qquad \mbox{in } D, \\
\sigma(\bfu_n) \bfnu + \i k A \bfu_n &= \mathbf{0},\qquad \mbox{on } \pa D,
\label{Eq:Nearly_Deterministic_2}
\end{align}
where 
\begin{align} \label{Eq:Nearly_Deterministic_3}
	\bfu_{-1} := \mathbf{0}, \qquad \bfS_0 := \bff, \qquad \bfS_n := 2k^2\eta 
	\bfu_{n-1} +k^2\eta^2\bfu_{n-2} \mbox{ for } n \geq 1.
\end{align}
As stated previously, the multiplicative structure of $\bfS_n$ does 
not allow computation of $\E(\bfS_n)$ directly for $n \geq 1$.  Thus, the mean of 
the mode function $\E(\bfu_n)$ cannot be computed directly for $n \geq 1$.  
Therefore, \eqref{Eq:Nearly_Deterministic_1}--\eqref{Eq:Nearly_Deterministic_2} is 
truly a random PDE system.  On the other hand, since all the coefficients in the 
equations \eqref{Eq:Nearly_Deterministic_1}--\eqref{Eq:Nearly_Deterministic_2} 
are constant, these SPDEs can be called ``nearly deterministic".  This property will 
be exploited in the same manner as \cite{Feng_Lin_Lorton_15} to develop an 
efficient numerical algorithm.

To compute $\E(\bfu_n)$, we use a discretization technique for the 
probability space $(\Omega,\cF,P)$.  There are many choices for such a 
discretization technique, but as noted in \cite{Feng_Lin_Lorton_15}, the Monte 
Carlo method is a good choice for ``nearly deterministic" PDEs such as 
\eqref{Eq:Nearly_Deterministic_1}--\eqref{Eq:Nearly_Deterministic_2}.  The Monte 
Carlo method will be combined with the interior penalty discontinuous Galerkin 
method given in \eqref{Eq:IPDG_Weak_Form_3} to produce a Monte Carlo interior 
penalty discontinuous Galerkin (MCIP-DG) method for approximating $\E(\bfu_n)$.

Following the standard formulation of the Monte Carlo method (cf. 
\cite{Babuska_Tempone_Zouraris_04}), let $M$ be a (large) positive integer which 
denotes the number of realizations used to generate the Monte 
Carlo approximation. For each $j = 1, \dots, M$ we sample i.i.d. realizations of 
the source term $\bff(\om_j,\cdot) \in \bL^2(D)$ and random media coefficient $
\eta(\omega_j, \cdot) \in L^{\infty}(D)$ such that $\| \eta(\omega_j,\cdot) \|
_{L^\infty(D)} \leq 1$.  With these realizations we recursively find 
corresponding approximations $\bfu_n^h(\om_j,\cdot) \in \bV^h$ such that

\begin{align}\label{Eq:Discrete_Mode}
	a_h \big( \bfu^h_n(\om_j,\cdot), \bfpsi^h \big) = \big( \bfS^h_n(\om_j,
	\cdot), \bfpsi^h \big)_{D} \qquad \forall \bfpsi \in \bV^h,
\end{align}
where 
\begin{align*}
	\bfu^h_{-1} := \mathbf{0}, \qquad \bfS^h_0 := \bff, \qquad \bfS^h_n := 
	2k^2\eta \bfu^h_{n-1} +k^2\eta^2\bfu^h_{n-2} \mbox{ for } n \geq 1,
\end{align*}
and $a_h(\cdot, \cdot)$ is defined in \eqref{Eq:IPDF_Weak_Form_2}. 
The MCIP-DG approximation $\bfPhi_n^h$ of $\E \big( \bfu_n \big)$ is defined as 
the following statistical average:
\begin{align} \label{Eq:Statistical_Average}
	\bfPhi^h_n := \frac{1}{M} \sum_{j = 1}^M \bfu^h_n(\om_j,\cdot).
\end{align}

The error associated with approximating $\E \big( \bfu_n \big)$ by its MCIP-DG 
approximation $\bfPhi_n^h$ can be decomposed in the following manner:
\begin{align*}
	\E \big( \bfu_n \big) - \bfPhi_n^h = \Big(\E \big( \bfu_n \big) - \E \big( 
	\bfu_n^h \big) \Big) + \Big(\E \big( \bfu_n^h \big) - \bfPhi_n^h \Big).
\end{align*}
To derive estimates on the error $\E \big( \bfu_n \big) - \E \big( 
	\bfu_n^h \big)$, we first establish stability estimates on $\bfu_n^h$.  
	These estimates are similar to those given in Theorem \ref{thm3.1} and are 
	given as the following lemma.
\begin{lemma} \label{lem:Discrete_Mode_Stability_Estimates}
	Assume that $k^{\ta + 1} h = O(1)$.  Then there holds for $n \geq 0$
	\begin{align}
		\E \big( \| \bfu^h_n \|_{L^2(D)}^2 \big) &\leq \hat{C}(n,k) \left( 
		k^{\ta - 2} + \frac{1}{k^2} \right)^2 \E \big( \| \bff \|^2_{L^2(D)} 
		\big), \label{eq:DMSE1} \\
		\E \big( \| \bfu^h_n \|_{1,h}^2 \big) &\leq \hat{C}(n,k) \left( k^{\ta - 
		1} + \frac{1}{k} \right)^2 \E \big( \| \bff \|^2_{L^2(D)} \big), 
		\label{eq:DMSE2}
	\end{align}
	where
	\begin{align} \label{eq:DMSE3}
		\hat{C}(0,k) := \hat{C}_0^2, \qquad \hat{C}(n,k) := 4^{2n-1} \hat{C}
		_0^{2n+2} (1 + k^{\ta})^{2n} \qquad \mbox{for } n \geq 1.
	\end{align}
\end{lemma}
\begin{proof}
For each $\om \in \Omega$, $\bfu^h_n(\om,\cdot)$ satisfies 
\eqref{Eq:Discrete_Mode}.  Thus, we apply \eqref{Eq:IPDG_Stability_2} and take 
the expectation to find
\begin{align} \label{eq:DMSE4}
\E\Bigl( \norm{\bfu^h_0}_{L^2(D)}^2  \Bigr)
&\leq \hat{C}^2_0 \Bigl( k^{\ta - 2} +\frac{1}{k^2} \Bigr)^2 \E(\norm{\bff}
_{L^2(D)}^2),\\
\E(\norm{\bfu^h_0}_{1,h}^2) &\leq \hat{C}^2_0 \Bigl(k^{\ta - 1} +\frac{1}{k^2} 
\Bigr)^2 \E(\norm{\bff}_{L^2(D)}^2). \label{eq:DMSE5} \\
\end{align}
Hence, \eqref{eq:DMSE1} and \eqref{eq:DMSE2} hold for $n=0$.

Next, we use induction to prove that \eqref{eq:DMSE1} and \eqref{eq:DMSE2} for all $n> 0$.  
Assume that \eqref{eq:DMSE1} and \eqref{eq:DMSE2} hold for 
all $0\leq n\leq \ell-1$, then again using \eqref{Eq:IPDG_Stability_2} and 
taking the expectation we get
\begin{align*}
&\E\Bigl( \norm{\bfu^h_\ell}_{L^2(D)}^2 \Bigr)
\leq 2\hat{C}^2_0\Bigl( k^{\ta - 2} +\frac{1}{k^2} \Bigr)^2
\E\Bigl( \norm{2k^2\eta \bfu^h_{\ell-1} }_{L^2(D)}^2
+\overline{\delta}_{1\ell} \norm{k^2\eta^2 \bfu^h_{\ell-2}}_{L^2(D)}^2 \Bigr) \\
& \qquad
\leq 2 \hat{C}^2_0\Bigl( k^{\ta - 2} +\frac{1}{k^2} \Bigr)^2
\big(1+k^{\ta}\big)^2 \Bigl( 4 \hat{C}(\ell-1,k) + \hat{C}(\ell-2,k) \Bigr) 
\E(\norm{\bff}_{L^2(D)}^2)\\
& \qquad
\leq \Bigl( k^{\ta - 2} +\frac{1}{k^2} \Bigr)^2 \, 8 \hat{C}^2_0 \big(1+k^{\ta}
\big)^2 \hat{C}(\ell-1,k) \left( 1+ \frac{\hat{C}(\ell-2,k)}{\hat{C}(\ell-1,k)} 
\right) \E(\norm{\bff}_{L^2(D)}^2)\\
& \qquad \leq \Bigl( k^{\ta - 2} +\frac{1}{k^2} \Bigr)^2 \hat{C}(\ell, k) 
\E(\norm{\bff}_{L^2(D)}^2),
\end{align*}
where $\overline{\delta}_{1\ell}=1-\delta_{1\ell}$ and $\delta_{1\ell}$ denotes
the Kronecker delta, and we have used the fact that $k \geq 1$ and
\[
8\hat{C}^2_0 \big(1+k^{\ta}\big)^2 \hat{C}(\ell-1,k) \left( 1+ \frac{\hat{C}
(\ell-2,k)}{\hat{C}(\ell-1,k)} \right) \leq \hat{C}(\ell, k).
\]
Similarly, we have
\begin{align*}
&\E\bigl(\norm{\bfu^h_\ell}_{1,h}^2 \bigr)
\leq 2 \hat{C}^2_0\Bigl( k^{\ta - 1} + \frac{1}{k^2} \Bigr)^2
\E\Bigl( \norm{2k^2\eta \bfu^h_{\ell-1} }_{L^2(D)}^2
+\overline{\delta}_{1\ell} \norm{k^2\eta^2 \bfu^h_{\ell-2}}_{L^2(D)}^2 \Bigr) \\
& \qquad\leq 2\hat{C}^2_0\Bigl(k^{\ta - 1} +\frac{1}{k^2} \Bigr)^2
\big(1+k^{\ta}\big)^2 \Bigl( 4 \hat{C}(\ell-1,k)+ \hat{C}(\ell-2,k) \Bigr) \E
\bigl(\norm{\bff}_{L^2(D)}^2\bigr)\\
& \qquad \leq \Bigl(k^{\ta - 1} +\frac{1}{k^2} \Bigr)^2 \hat{C}(\ell, k) \E
\bigl(\norm{\bff}_{L^2(D)}^2\bigr).
\end{align*}
Hence, \eqref{eq:DMSE1} and \eqref{eq:DMSE2} hold for $n=\ell$ and the induction 
argument is complete.
\end{proof}

Therefore, to prove estimates for the error $\bfu_n-\bfu_n^h$, it is important to note 
that in order to ensure $\bfu^h_n(\om_j,\cdot)$ is computable, the discrete right-hand 
source term $\bfS^h_n(\om_j,\cdot)$ is used in place of $\bfS_n(\om_j,\cdot)$.  
To account for this change we introduce auxiliary mode function $\bftu^h_n$ 
which satisfies the following equation:
\begin{align}\label{Eq:Aux_Mode}
	a_h \big( \bftu^h_n(\om_j,\cdot), \bfpsi^h \big) = \big( \bfS_n(\om_j,
	\cdot), \bfpsi^h \big)_{D} \qquad \forall \bfpsi \in \bV^h,
\end{align}
for each realization $\om_j$ and $n \geq 0$. The auxiliary function 
$\bftu_n^h$ as well as the following technical lemma from 
\cite{Feng_Lin_Lorton_15} are used to prove the desired error estimate.
\begin{lemma}\label{lem:Simp}
Let $\gamma,\beta > 0$ be two real numbers, $\{c_n\}_{n\geq 0}$ and
$\{\alpha_n\}_{n\geq 0}$ be two sequences of nonnegative numbers such that
\begin{equation}\label{eq:Simp_1}
c_0\leq \gamma\alpha_0, \quad
c_n\leq \beta c_{n-1} +\gamma \alpha_n \quad \mbox{for } n\geq 1.
\end{equation}
Then there holds
\begin{equation}\label{eq:Simp_2}
c_n\leq \gamma \sum_{j=0}^n \beta^{n-j} \alpha_j \qquad\mbox{for } n\geq 1.
\end{equation}
\end{lemma}

Now we are able to prove the following theorem.
\begin{theorem} \label{thm:IPDG_Mode_Error}
Suppose that $k^{\ta + 1} h = O(1)$, then there hold
\begin{align}
	\label{Eq:IPDG_Mode_Error_1} 
	&\E \big( \| \bfu_n - \bfu_n^h \|_{L^2(D)} \big) 
	 \leq \tilde{C}_0 h^2 \sum_{j = 0}^n\big[ \hat{C}_0 (2k^{\ta} + 3) 
	\big]^{n-j} \E \big( \| \bfu_j \|_{H^2(D)} \big),  \\
	&\E \big( \| \bfu_n - \bfu_n^h \|_{1,h} \big) \label{Eq:IPDG_Mode_Error_2} 
        \leq C \tilde{C}_0 h \sum_{j = 0}^n \big[ \hat{C}_0 (2k^{\ta} + 3) 
	\big]^{n-j} \E \big( \| \bfu_j \|_{H^2(D)} \big), 
\end{align}
where $C$, $\tilde{C}_0$, and $\hat{C}_0$ are constants independent of $k$ and $h$.
\end{theorem}

\begin{proof}
We consider the following error decomposition:
\begin{align*}
	\bfu_n - \bfu_n^h = \big(\bfu_n - \bftu_n^h \big) + \big( \bftu_n^h - 
	\bfu_n^h \big),
\end{align*}
for $n = 0, 1, 2, \cdots$. Using the 
fact that $k^{\ta + 1} h = O(1)$ and applying Theorem \ref{thm:IPDG_Error} part (ii) yield
\begin{align}
	\E \big( \| \bfu_n - \bftu_n^h \|_{L^2(D)} \big) & \leq \tilde{C}_0 h^2 \E 
	\big( \| \bfu_n \|_{H^2(D)} \big), \label{Eq:IDPG_Mode_Error_Temp_1} \\
	\E \big( \| \bfu_n - \bftu_n^h \|_{1,h} \big) & \leq \tilde{C}_0 h  \E \big( 
	\| \bfu_n \|_{H^2(D)} \big). \label{Eq:IDPG_Mode_Error_Temp_2}
\end{align}
Subtracting \eqref{Eq:Discrete_Mode} from \eqref{Eq:Aux_Mode} yields
\begin{align*}
	a_h \big( \bftu^h_n - \bfu^h_n, \bfpsi^h \big) = \big( \bfS_n - \bfS^h_n, 
	\bfpsi^h \big)_{D} \qquad \forall \bfpsi \in \bV^h.
\end{align*}
	Thus, it follows from Theorem \ref{thm:IPDG_Stability} part (ii) that
\begin{align} \label{Eq:IDPG_Mode_Error_Temp_3}
	&\E \big( \| \bftu^h_n - \bfu^h_n \|_{L^2(D)} \big) 
	\leq \hat{C}_0 \left(k^{\ta - 2} + \frac{1}{k^2} \right) \E \big( 
	\| \bfS_n - \bfS^h_n \|_{L^2(D)} \big) \\
	& \qquad \leq 2 \hat{C}_0 \big(k^{\ta} + 1 \big) \Big( \E \big( \| 
	\bfu_{n-1} - \bfu_{n-1}^h \|_{L^2(D)} \big) + \E \big( \| \bfu_{n-2} - 
	\bfu_{n-2}^h \|_{L^2(D)} \big) \Big), \notag
\end{align}
where we define $\bfu_{-1} = \bfu_{-2} = \bfu_{-1}^h = \bfu_{-2}^h = 0$. By 
making the simplifying assumption that $\hat{C}_0 \geq 1$ and combining 
\eqref{Eq:IDPG_Mode_Error_Temp_1} with \eqref{Eq:IDPG_Mode_Error_Temp_3} we get
\begin{align*}
&\E \big( \| \bfu_n - \bfu_n^h \|_{L^2(D)} \big) + \E \big( \| \bfu_{n-1} - 
\bfu_{n-1}^h \|_{L^2(D)} \big)  \\
& \qquad \leq \E \big( \| \bfu_n - \bftu_n \|_{L^2(D)} \big) + \E \big( \| 
\bftu_n - \bfu_n^h \|_{L^2(D)} \big) + \E \big( \| \bfu_{n-1} - \bfu_{n-1}^h \|
_{L^2(D)} \big) \\
& \qquad \leq 2 \hat{C}_0 \big(k^{\ta} + 1 \big) \Big( \E \big( \| \bfu_{n-1} - 
\bfu_{n-1}^h \|_{L^2(D)} \big) + \E \big( \| \bfu_{n-2} - \bfu_{n-2}^h \|
_{L^2(D)} \big) \Big) \\
& \qquad \qquad + \tilde{C}_0 h^2 \E \big( \| \bfu_n \|_{H^2(D)} \big) + \E 
\big( \| \bfu_{n-1} - \bfu_{n-1}^h \|_{L^2(D)} \big) \\
& \qquad \leq \hat{C}_0 \big(2k^{\ta} + 3\big) \Big( \E \big( \| \bfu_{n-1} - 
\bfu_{n-1}^h \|_{L^2(D)} \big) + \E \big( \| \bfu_{n-2} - \bfu_{n-2}^h \|
_{L^2(D)} \big) \Big) \\
& \qquad \qquad + \tilde{C}_0 h^2 \E \big( \| \bfu_n \|_{H^2(D)} \big).
\end{align*}
Then, by applying Lemma \ref{lem:Simp} with
\begin{align*}
	&c_n := \E \big( \| \bfu_n - \bfu_n^h \|_{L^2(D)} \big) + \E \big( \| 
	\bfu_{n-1} - \bfu_{n-1}^h \|_{L^2(D)} \big) \\
	&\beta := \hat{C}_0 \big(2k^{\ta} + 3\big), \qquad \gamma := \tilde{C}_0 
	h^2, \qquad \alpha_n := \E \big( \| \bfu_n \|_{L^2(D)} \big),
\end{align*}
we arrive at \eqref{Eq:IPDG_Mode_Error_1}.

To obtain \eqref{Eq:IPDG_Mode_Error_2}, we apply the inverse inequality along 
with \eqref{Eq:IDPG_Mode_Error_Temp_1}, \eqref{Eq:IDPG_Mode_Error_Temp_2}, and 
\eqref{Eq:IPDG_Mode_Error_1} in the following manner:
\begin{align*}
	&\E \big( \| \bfu_n - \bfu_n^h \|_{1,h} \big) \leq \E \big( \| \bfu_n - 
	\bftu_n^h \|_{1,h} \big) + \E \big( \| \bftu_n^h - \bfu_n^h \|{1,h} \big) \\
	& \qquad \leq \E \big( \| \bfu_n - \bftu_n^h \|_{1,h} \big) + C h^{-1} \E 
	\big( \| \bftu_n^h - \bfu_n^h \|{L^2(D)} \big) \\
	& \qquad \leq \E \big( \| \bfu_n - \bftu_n^h \|_{1,h} \big) + C h^{-1} \E 
	\big( \| \bftu_n^h - \bfu_n \|{L^2(D)} \big) + C h^{-1} \E \big( \| \bfu_n - 
	\bfu_n^h \|{L^2(D)} \big) \\
	& \qquad \leq \tilde{C}_0 h \E \big( \| \bfu_n \|_{H^2(D)} \big) + C 
	\tilde{C}_0 h \E \big( \| \bfu_n \|_{H^2(D)} \big) \\
	& \qquad \qquad + C \tilde{C}_0 h \sum_{j = 0}^n \Big[\hat{C}_0 
	\big(2k^{\ta} + 3 \big) \Big]^{n-j} \E \big( \| \bfu_j \|_{H^2(D)} \big).
\end{align*}
Thus, \eqref{Eq:IPDG_Mode_Error_2} holds.  The proof is complete.
\end{proof}

To estimate the error associated with approximating $\E \big( 
\bfu_n^h \big)$ by its Monte Carlo approximation $\bfPhi_n^h$, we 
use the following well-known lemma (cf.  \cite{Babuska_Tempone_Zouraris_04, Liu_Riviere_13}).
\begin{lemma}\label{lem4.1}
There hold the following estimates for $n\geq 0$
\begin{align}\label{eq4.27}
\E\bigl(\|\E(\bfu^h_n) -\bfPhi^h_n\|_{L^2(D)}^2 \bigr)
&\leq \frac{1}{M} \E(\|\bfu^h_n\|_{L^2(D)}^2),\\
\E\bigl(\|\E(\bfu^h_n) -\bfPhi^h_n\|_{1,h,D}^2 \bigr)
&\leq \frac{1}{M} \E(\|\bfu^h_n\|_{1,h}^2). \label{eq4.27a}
\end{align}
\end{lemma}

Lemma \ref{lem4.1} and \ref{lem:Discrete_Mode_Stability_Estimates} are 
combined to give an estimate for the error associated with approximating $\E \big( 
\bfu^h_n \big)$ with $\bfPhi^h_n$.
\begin{theorem} \label{thm:Mode_Average_Error}
Suppose that $k^{\ta + 1} h = O(1)$, then there hold
\begin{align}\label{eq:Mode_Average_Error1}
\E\bigl(\|\E(\bfu^h_n) -\bfPhi^h_n\|_{L^2(D)}^2 \bigr)
&\leq \frac{1}{M} \Bigl(k^{\ta - 2} +\frac{1}{k^2}\Bigr)^2
\hat{C}(n,k)\, \E(\|\bff\|_{L^2(D)}^2),\\
\E\bigl(\|\E(\bfu^h_n) -\bfPhi^h_n\|_{1,h}^2 \bigr)
&\leq \frac{1}{M} \Bigl(k^{\ta - 1} + \frac{1}{k^2}\Bigr)^2
\hat{C}(n,k)\, \E(\|\bff\|_{L^2(D)}^2). \label{eq:Mode_Average_Error2}
\end{align}
\end{theorem}

\section{The overall numerical procedure} \label{sec:Numerical_Procedure}
This section is devoted to presenting an efficient algorithm 
for approximating the mean of the solution to the random elastic Helmholtz 
problem \eqref{Eq:ElasticPDE}--\eqref{Eq:ElasticBC}.  The efficiency of the 
algorithm relies heavily on the \textit{multi-modes} expansion of the solution 
given in \eqref{eq:MultiModes}.  This section also gives a comprehensive 
convergence analysis for the proposed multi-modes MCIP-DG method.

\subsection{The numerical algorithm, linear solver, and computational complexity}

This subsection describes our multi-modes MCIP-DG algorithm as well as its 
computational complexity.  We demonstrate that the new multi-modes MCIP-DG 
algorithm has a better computational complexity than the classical MCIP-DG 
method applied to the random elastic Helmholtz problem  
\eqref{Eq:ElasticPDE}--\eqref{Eq:ElasticBC}.  Classical MCIP-DG method refers to the 
MCIP-DG method applied to the problem without making use of the multi-modes 
expansion of the solution.

First, we state the classical MCIP-DG method for approximating $\E(\bfu)$. For 
the rest of this paper $\bftPsi^h$ will refer to the classical MCIP-DG 
approximation generated by Algorithm 1 below.  To state Algorithm 1 we define 
the following sesquilinear form: 
\begin{align*}
	\hat{a}_j^h(\bfphi, \bfpsi) &:= b_h(\bfphi, \bfpsi) - k^2 (\alpha^2(\om_j,\cdot) \bfphi,\bfpsi)_D + \i k 
	\langle \alpha(\om_j,\cdot)  A \bfphi, \bfpsi \rangle_{\pa D} \\
	& \qquad+ \i \big( J_0 (\bfphi, \bfpsi) + 
	J_1(\bfphi, \bfpsi) \big).
\end{align*}
 
\noindent {\bf Algorithm 1 (Classical MCIP-DG)}

\begin{description}
\item Input $\bff, \eta, \veps, k, h, M.$
\item Set $\bftPsi^h(\cdot)= \mathbf{0}$ (initializing).
\begin{description}
\item For $j=1,2,\cdots, M$
\item Obtain realizations $\eta(\om_j,\cdot)$ and $\bff(\om_j,\cdot)$.
\item Solve for $\hat{\bfu}^h(\omega_j,\cdot) \in \bV^h$ such that
\[
\hat{a}^h_j\bigl( \hat{\bfu}^h(\omega_j,\cdot), \bfv_h \bigr) = 
\bigl(\bff(\omega_j,\cdot), \bfv_h\bigr)_D \qquad\forall v_h\in \bV^h.
\]
\item Set $\bftPsi^h(\cdot) \leftarrow \bftPsi^h(\cdot) +\frac{1}{M} \hat{\bfu}
^h(\omega_j,\cdot)$.
\item Endfor
\end{description}
\item Output $\bftPsi^h(\cdot)$.
\end{description}

\smallskip
For convergence of the Monte Carlo method, the number of realizations $M$ must 
be sufficiently large.  Thus, one must solve a large number of deterministic 
elastic Helmholtz problems when implementing the classical MCIP-DG method. In 
the case that the frequency $k$ is taken to be large, solving a deterministic 
elastic Helmholtz problem equates to solving a large, ill-conditioned, and 
indefinite linear system.  It is well known that standard iterative methods do 
not perform well for Helmholtz-type problems \cite{Ernst_Gander_12}.  For this 
reason, Gaussian elimination is considered to solve each linear system in the 
internal for-loop of Algorithm 1.  Such a large number of Gaussian elimination 
solves will make Algorithm 1 impractical.

To eliminate the need for performing many Gaussian elimination steps, we leverage 
the multi-modes expansion of the solution \eqref{eq:MultiModes} and propose
the following multi-modes algorithm for \eqref{Eq:ElasticPDE}--\eqref{Eq:ElasticBC}:

\smallskip
\noindent
{\bf Algorithm 2 (Multi-Modes MCIP-DG)}
\begin{description}
\item Input $\bff, \eta, \veps, k, h, M,N$
\item Set $\bfPsi^h_N(\cdot)=0$ (initializing).
\item Generate the stiffness matrix $A$ from the sesquilinear form 
$a_h(\cdot,\cdot)$ on $\bV^h \times \bV^h$.
\item Compute and store the $LU$ decomposition of $A$.
\begin{description}
\item For $j=1,2,\cdots, M$
\item Obtain realizations $\eta(\om_j,\cdot)$ and $\bff(\om_j,\cdot)$.
\item Set $\bfS^h_0(\omega_j,\cdot)=\bff(\omega_j,\cdot)$.
\item Set $\bfu^h_{-1}(\omega_j,\cdot)=0$.
\item Set $\bfU^h_N(\omega_j,\cdot)=0$ (initializing).
\begin{description}
\item For $n=0,1,\cdots, N-1$
\item Solve for $\bfu^h_n(\omega_j,\cdot) \in \bV^h$ such that
\[
a_h\bigl( \bfu^h_n(\omega_j,\cdot), \bfv_h \bigr) = \bigl(\bfS^h_n(\omega_j,
\cdot), \bfv_h\bigr)_D \qquad\forall \bfv_h\in \bV^h,
\]
using the $LU$ decomposition of $A$.
\item Set $\bfU^h_N(\omega_j,\cdot)\leftarrow \bfU^h_N(\omega_j,\cdot) +\veps^n 
\bfu^h_n(\omega_j,\cdot)$.
\item Set $\bfS^h_{n+1}(\omega_j,\cdot)=2k^2 \eta(\omega_j,\cdot) 
\bfu^h_n(\omega_j,\cdot)
+ k^2 \eta(\omega_j,\cdot)^2 \bfu^h_{n-1}(\omega_j,\cdot)$.
\item Endfor
\end{description}
\item Set $\bfPsi^h_N(\cdot) \leftarrow \bfPsi^h_N(\cdot) +\frac{1}{M} 
\bfU^h_N(\omega_j,\cdot)$.
\item Endfor
\end{description}
\item Output $\bfPsi^h_N(\cdot)$.
\end{description}

\smallskip
We note that $\bfPhi^h_n$, defined in \eqref{Eq:Statistical_Average} does not 
show up explicitly in Algorithm 2.  Instead, the multi-modes MCIP-DG approximation 
$\bfPsi^h_N$ is related to $\{\bfPhi^h_n\}$ by
\begin{align}
	\bfPsi^h_N = \sum_{n = 0}^{N-1} \veps^n \bfPhi_n^h. \label{Eq:5_1}
\end{align}
This relationship will be used to obtain the convergence analysis presented in the 
next subsection.

To compare the efficiency of Algorithm 2 versus Algorithm 1, 
let $L = \frac{1}{h}$ with $h$ be the mesh size.  As was stated in 
\cite{Feng_Lin_Lorton_15}, Algorithm 1 requires $O \big( ML^{3d} \big)$ 
multiplications and Algorithm 2 requires $O \big( L^{3d} + MNL^{2d} \big)$ 
multiplications.  In practice, the number of modes $N$ is 
relatively small (see Theorem \ref{thm:Total_Error}) so we can treat it as a 
constant.  To achieve equal order in the $L^2$-error associated to the IP-DG 
method as well as the error associated to the Monte Carlo method, one can choose 
$M = L^4$.  In this case the number of multiplications used in Algorithm 1 is 
given by $O \big( L^{3d + 4} \big)$, where as the number of multiplications used 
in Algorithm 2 is $O \big( L^{3d} + L^{2d + 4} \big)$.  Thus Algorithm 2 is much
more efficient than Algorithm 1.

It is well known that the Monte Carlo algorithm is naturally parallelizable.  
The outer for-loop in both Algorithm 1 and 2 can be run in parallel.

\subsection{Convergence analysis} \label{subsel:convergence}
This subsection provides estimates for the total error, $\E (\bfu^\veps) - 
\bfPsi^h_N$, associated to Algorithm 2.  To this end, we introduce the following 
error decomposition:
\begin{align} \label{Eq:Error_Decomposition}
	&\E(\bfu^\veps)-\bfPsi^h_N \\
	& \qquad =\bigl(\E(\bfu^\veps)-\E(\bfU^\veps_N)\bigr) + 
	\bigl( \E(\bfU^\veps_N)- \E(\bfU^h_N)\bigr) +\bigl( \E(\bfU^h_N)-\bfPsi^h_N 
	\bigr), \notag
\end{align}
where $\bfU^\veps_N$ is defined in $\eqref{eq3.12}$ and $\bfU^h_N$ is defined as
\begin{align}
	\bfU^h_N = \sum_{n = 0}^{N-1} \bfu_n^h \label{Eq:Truncated_MM_IPDG}.
\end{align}
The first term on the right-hand side of \eqref{Eq:Error_Decomposition} 
corresponds the error associated to truncating the multi-modes expansion on 
$\bfu^\veps$, the second term corresponds to the error associated to using 
the IP-DG discretization method, and the third term corresponds to the error 
associated to the Monte Carlo method.  The error corresponding to truncation 
of the multi-modes expansion was estimated in Theorem \ref{thm3.3}.

The definition of $\bfU_N^\veps$ and $\bfU^h_N$ and Theorem 
\ref{thm:IPDG_Mode_Error} immediately implies the following theorem 
characterizing the error associated to the IP-DG method.
\begin{theorem} \label{thm:MM_IPDG_Error_1}
Assume that $\bfu_n \in \bL^2(\Omega, \bH^2(D))$ for $n \geq 0$. If $k^{\ta + 1} 
h = O(1)$, then the following error estimates hold
\begin{align} \label{Eq:MM_IPDG_Error_1} 
	&\E \big( \| \bfU_N^\veps - \bfU^h_N \|_{L^2(D)} \big) 
	\leq \tilde{C}_0 h^2 \sum_{n = 0}^{N-1} \sum_{j = 0}^n \veps^n 
	\big[ \hat{C}_0 (2k^{\ta} + 3) \big]^{n-j} \E \big( \| \bfu_j \|_{H^2(D)} \big), \\
	&\E \big( \bfU_N^\veps - \bfU^h_N  \|_{1,h} \big) \label{Eq:MM_IPDG_Error_2} 
\leq C \tilde{C}_0 h \sum_{n=0}^{N-1} \sum_{j = 0}^n \veps^n \big[ 
	\hat{C}_0 (2k^{\ta} + 3) \big]^{n-j} \E \big( \| \bfu_j \|_{H^2(D)} \big). 
\end{align}
\end{theorem}

As an immediate consequence of this theorem, we choose suitable restrictions on 
the size of the perturbation parameter $\veps$ and obtain the following theorem.
\begin{theorem} \label{thm:MM_IPDG_Error_2}
Assume that $\bfu_n \in \bL^2(\Omega, \bH^2(D))$ for $n \geq 0$. If $k^{\ta + 1} 
h = O(1)$ and $\veps$ is chosen to satisfy $4 \hat{C}_0 \sqrt{C_0} \big(2k^{\ta} 
+3\big) \veps < 1$, then there hold
\begin{align}
	\E \big( \| \bfU_N^\veps - \bfU^h_N \|_{L^2(D)} \big) & \leq C(C_0,\hat{C}
	_0, \tilde{C}_0,k,\veps) \, h^2, \label{Eq:MM_IPDG_Error_3} \\
	\E \big( \bfU_N^\veps - \bfU^h_N  \|_{1,h} \big) &\leq C(C_0,\hat{C}_0, 
	\tilde{C}_0,k,\veps) \, h, \label{Eq:MM_IPDG_Error_4}
\end{align}
where
\begin{align*}
	C(C_0,\hat{C}_0, \tilde{C}_0,k,\veps):= C \, \tilde{C}_0 \frac{\sqrt{C_0}
	\big(k^{\ta + 2} + 1\big)}{2k^2 \big(4\sqrt{C_0} - 1\big)} \cdot \frac{1}{1 
	- 4 \hat{C}_0 \sqrt{C_0} \big(2k^{\ta} + 3\big)\veps}.
\end{align*}
\end{theorem}

\begin{proof}
To obtain \eqref{Eq:MM_IPDG_Error_3} and \eqref{Eq:MM_IPDG_Error_4}, we need to 
find an upper bound for the double sum in \eqref{Eq:MM_IPDG_Error_1} and 
\eqref{Eq:MM_IPDG_Error_2}.  To this double sum we apply \eqref{eq3.6a} and make 
the simplifying assumption that $C_0 \geq 1$ to obtain the following
\begin{align*}
	&\sum_{n = 0}^{N-1} \sum_{j = 0}^n \veps^n \big[ \hat{C}_0 (2k^{\ta} + 3) 
	\big]^{n-j} \E \big( \| \bfu_j \|_{H^2(D)} \big) \\
	& \qquad \leq \left(k^{\ta} + \frac{1}{k^2} \right) \E \big(\| \bff \|
	_{L^2(D)} \big) \sum_{n = 0}^{N-1} \sum_{j = 0}^n \veps^n \big( \hat{C}_0 
	(2k^{\ta} + 3) \big)^{n-j} C(j,k)^{\frac{1}{2}} \\
	& \qquad \leq \frac{\sqrt{C_0}\big(k^{\ta + 2} + 1\big)}{2k^2} \E \big(\| 
	\bff \|_{L^2(D)} \big) \sum_{n = 0}^{N-1}  \veps^n \big( \hat{C}_0 (2k^{\ta} 
	+ 3) \big)^{n}  \sum_{j = 0}^n 4^j C_0^{\frac{j}{2}} 
\end{align*}
\begin{align*}
	& \qquad \leq \frac{\sqrt{C_0}\big(k^{\ta + 2} + 1\big)}{2k^2 \big(4 
	\sqrt{C_0} -1 \big)} \E \big(\| \bff \|_{L^2(D)} \big) \sum_{n = 0}^{N-1} 
	\big(4 \hat{C}_0 \sqrt{C_0} (2k^{\ta} + 3) \veps)^{n} \\
	& \qquad \leq \frac{\sqrt{C_0}\big(k^{\ta + 2} + 1\big)}{2k^2 \big(4 
	\sqrt{C_0} -1 \big)} \cdot \frac{1 - \big(4 \hat{C}_0 \sqrt{C_0} (2k^{\ta} + 
	3) \veps)^{N}}{1 - 4 \hat{C}_0 \sqrt{C_0} (2k^{\ta} + 3) \veps} \E \big(\| 
	\bff \|_{L^2(D)} \big).
\end{align*}
Appealing to the fact that $4 \hat{C}_0 \sqrt{C_0} \big(2k^{\ta} +3\big) \veps < 1$, 
yields \eqref{Eq:MM_IPDG_Error_3} and \eqref{Eq:MM_IPDG_Error_4}.
\end{proof}

The next theorem establishes the error associated to the Monte Carlo discretization.
\begin{theorem} \label{thm:MM_Monte_Carlo_Error}
	Let $k^{\ta + 1} h = O(1)$ and let $\veps$ satisfy $\hat{c}_{\veps} := 4 
	\hat{C}_0 \big(k^{\ta} + 1 \big) \veps < 1$, then the following error estimates hold
	\begin{align}
		\E \Big( \| \E \big( \bfU_N^h \big) - \bfPsi_N^h \|_{L^2(D)} \Big) & 
		\leq \frac{\hat{C}_0}{\sqrt{M}} \left(k^{\ta - 2} + \frac{1}{k^2} 
		\right) \cdot \frac{1}{1 - \hat{c}_{\veps}} \E \big( \| \bff \|_{L^2(D)} 
		\big), \label{Eq:MM_Monte_Carlo_Error_1} \\
		\E \Big( \| \E \big( \bfU_N^h \big) - \bfPsi_N^h \|_{1,h} \Big) & \leq 
		\frac{\hat{C}_0}{\sqrt{M}} \left(k^{\ta - 1} + \frac{1}{k} \right) \cdot 
		\frac{1}{1 - \hat{c}_{\veps}} \E \big( \| \bff \|_{L^2(D)} \big). 
		\label{Eq:MM_Monte_Carlo_Error_2}
	\end{align}
\end{theorem}

\begin{proof}
	By the definitions of $\bfU_N^h$ and $\bfPsi_N^h$ we have
	\begin{align*}
		\bfU_N^h - \bfPsi_N^h = \sum_{n = 0}^{N-1} \veps^n \big( \bfu_n^h-\bfPhi_n^h \big) 
	\end{align*}
	Thus, it follows from \eqref{eq:Mode_Average_Error1} that
	\begin{align*}
		&\E \Big( \| \E \big( \bfU_N^h \big) - \bfPsi_N^h \|_{L^2(D)} \Big) \leq 
		\sum_{n = 0}^{N-1} \veps^n \E \Big( \| \E \big( \bfu_n^h \big) - 
		\bfPhi_n^h \|_{L^2(D)} \Big) \\
		& \qquad \leq \frac{1}{\sqrt{M}} \left(k^{\ta - 2} + \frac{1}{k^2} 
		\right) \E \big( \| \bff \|_{L^2(D)} \big) \sum_{n = 0}^{N-1} \veps^n 
		\hat{C}(n,k)^{\frac{1}{2}} \\
		& \qquad \leq \frac{\hat{C}_0}{\sqrt{M}} \left(k^{\ta - 2} + \frac{1}
		{k^2} \right) \E \big( \| \bff \|_{L^2(D)} \big) \sum_{n = 0}^{N-1} 
		\big( 4 \hat{C}_0 (1 + k^{\ta}) \veps \big)^n \\
		& \qquad \leq \frac{\hat{C}_0}{\sqrt{M}} \left(k^{\ta - 2} + \frac{1}
		{k^2} \right) \cdot \frac{1}{1 - \hat{c}_{\veps}} \E \big( \| \bff \|
		_{L^2(D)} \big).
	\end{align*}
	Hence, \eqref{Eq:MM_Monte_Carlo_Error_1} holds.  By using a similar argument 
to the one for deriving \eqref{eq:Mode_Average_Error2}, we obtain 
\eqref{Eq:MM_Monte_Carlo_Error_2}. The proof is complete.
\end{proof}

Theorems \ref{thm3.3}, \ref{thm:MM_IPDG_Error_2}, \ref{thm:MM_Monte_Carlo_Error} 
contain estimates for each piece of the total error decomposition given in 
\eqref{Eq:Error_Decomposition}.  The next theorem puts these together to give an 
estimate for the total error associated to the multimodes MCIP-DG method given 
in Algorithm 2.
\begin{theorem} \label{thm:Total_Error}
Under the assumptions that $\bfu_n \in \bL^2(\Omega, \bH^2(D))$ for $n \geq 0$, 
$k^{\ta + 1}h = O(1)$, and $4 \hat{C}_0 \sqrt{C_0} \big(2k^{\ta} +3\big) \veps < 1$, 
there hold
\begin{align}
	\E \Big( \| \E \big(\bfu^\veps \big) - \bfPsi_N^h \|_{L^2(D)} \Big) & \leq 
	C_1 \veps^N + C_2 h^2 + C_3 M^{-\frac{1}{2}}, \label{Eq:Total_Error_1} \\
	\E \Big( \| \E \big(\bfu^\veps \big) - \bfPsi_N^h \|_{1,h} \Big) & \leq C_1 
	\veps^N + C_2 h + C_3 M^{-\frac{1}{2}}, \label{Eq:Total_Error_2}
\end{align}
where $C_j = C_j(C_0, \hat{C}_0, \tilde{C}_0, k, \veps)$ are positive constants.
\end{theorem}

\section{Numerical experiments} \label{sec:Numerical_Experiments}

In this section we present several numerical experiments to demonstrate the 
performance and key features of the proposed MCIP-DG method for problem 
\eqref{Eq:ElasticPDE}--\eqref{Eq:ElasticBC}.  In all of our experiments the 
computational domain $D$ is chosen to be the unit square centered at the origin  
and a uniform triangulation of $D$ is chosen as the mesh.  A 
sample triangulation is given in Figure \ref{fig:MeshExample}.

\begin{figure}[htbp]
\centering
\includegraphics[width=3.5in,height=1.8in]{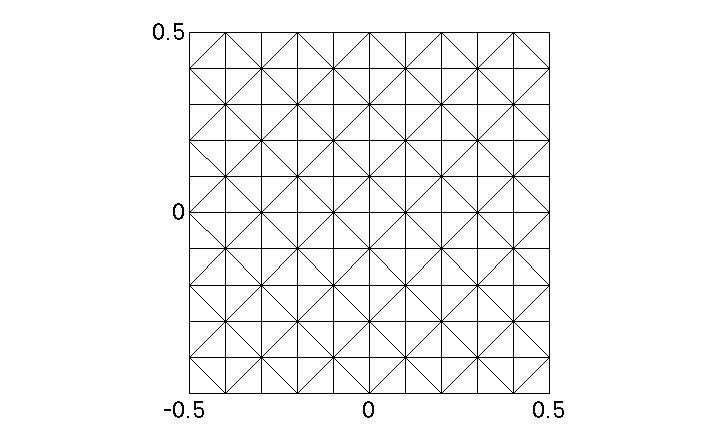} 
\caption{Triangulation $\mathcal{T}_{1/10}$ \label{fig:MeshExample}} 
\end{figure}

To simulate random media held inside $D$, $\alpha(\omega,\cdot)$ is generated 
using a random field $\eta(\omega,\cdot)$ which is a Gaussian random field with 
an exponential covariance function with correlation length $\ell = 0.5$ 
(cf. \cite{Lord_Powell_Shardlow}), i.e. the following covariance function 
is used to generate $\eta(\omega,\cdot)$
\begin{align*}
	C(\mathbf{x_1},\mathbf{x_2}) = \exp \left( -\frac{\| \mathbf{x_1}-
	\mathbf{x_2} \|_2}{0.5} \right).
\end{align*}
For the numerical simulations $\eta(\omega,\cdot)$ is sampled at the center 
of each element in the mesh. Such two sample realizations of $\eta(\omega,\cdot)$ 
are given in Figure \ref{fig:SmoothNoiseSamples}.
\begin{figure}[htbp]
\centerline{
\includegraphics[width=2.2in,height=1.8in]{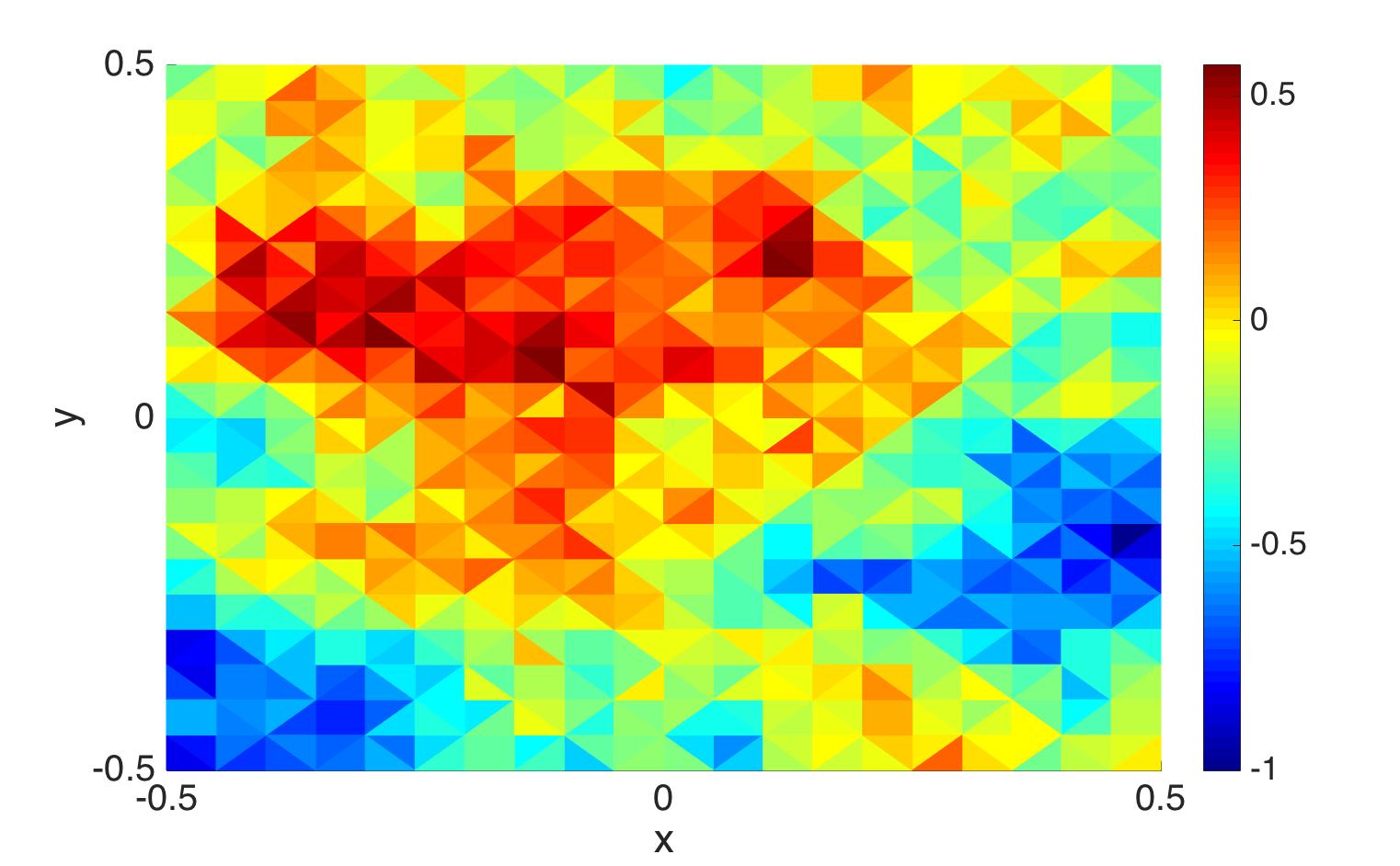} 
\includegraphics[width=2.2in,height=1.8in]{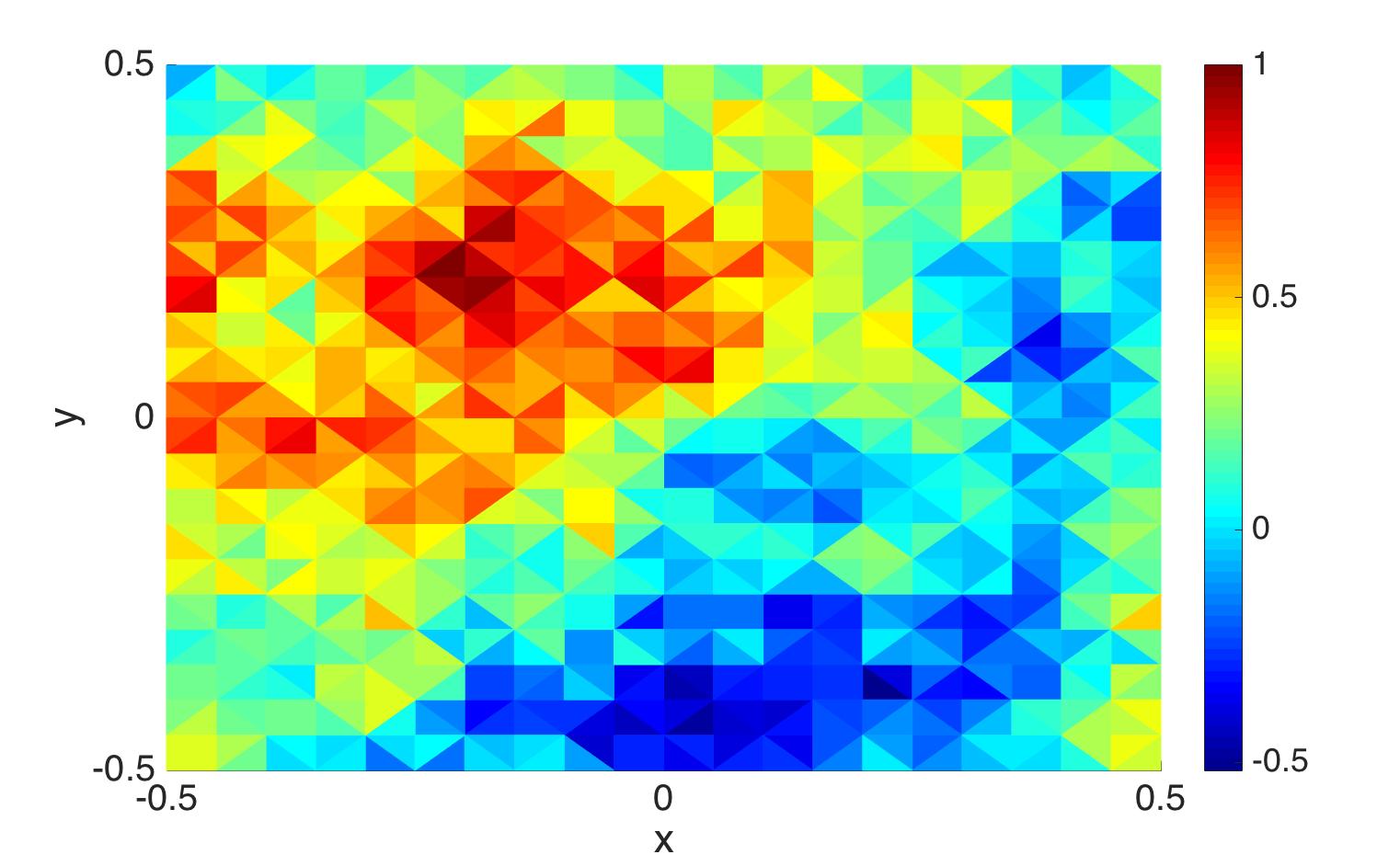}}
\caption{Samples of the random field $\eta(\omega,\cdot)$ generated using an 
exponential covariance function with covariance length $l = 0.5$ on a triangulation 
parameterized by $h = 1/20$. } \label{fig:SmoothNoiseSamples}
\end{figure}

Due to the difficulty of generating a test problem with a known solution, 
we use a contrived right-hand side source function of the form
\begin{align*}
	\bff = \frac{1}{(k \alpha(\om,\cdot))^2 r}[e^{\bfi (k \alpha(\om,\cdot)) r} - 1, 
e^{-\bfi (k \alpha(\om,\cdot))r} - 1]^T,
\end{align*}
where $r = | \bfx |$ is the Euclidean length of $\bfx$.  As a baseline, we 
compare the approximation $\bfPsi^h_N$ generated by Algorithm 2 to the 
classical Monte Carlo approximation $\tilde{\bfPsi}^h$ generated by Algorithm 1.

In our experiments the frequency $k$, the perturbation parameter $\veps$, and 
the number of modes $N$ are allowed to vary.  Other parameters are given the 
following values:
\begin{align*}
	h = 1/20, \qquad \mu = 1, \qquad \lambda = 1, \qquad M = 1000, \qquad A = 
	\left[\begin{array}{c c}
		1 & 0 \\
		0 & 1
	\end{array}\right].
\end{align*}
Based on the numerical experiments carried out in \cite{Feng_Lorton_15} we use 
the following values for our penalty parameters:
\begin{align*}
	\gamma_{0,e} = 10, \qquad \gamma_{1,e} = 0.1,
\end{align*}
for all edges $e$.

In our first experiment we want to judge the performance of Algorithm 2 when 
$\veps$ is taken to be small relative to the frequency $k$.  For this reason 
the frequency $k$ is set to be $k = 5$ and the perturbation parameter $
\veps$ is chosen to be $\veps = 1/10$. We expect from Theorem 
\ref{thm:MM_Monte_Carlo_Error} that the relative error between $\bfPsi^h_N$ and 
$\tilde{\bfPsi}^h$ should decrease on the order $\veps^N$.  

The relative error is plotted versus the $5\veps^N$ for $N = 1,2,\cdots,7$ in 
Figure \ref{fig:ModesVsFullMonte}.  A log scale is used on the vertical access 
to compare the two.  From this plot we note that the relative error is small 
for all values of $N$, indicating that $\bfPsi^h_N$ agrees with $\tilde{\bfPsi}
^h$ and thus Algorithm 2 is producing accurate results.  We also observe that 
the relative error decreases at approximately the same order as $\veps^N$.  
What is unexpected in this plot is the fact that the relative error is almost 
constant between $N = 1$ and $N = 2$, $N = 3$ and $N = 4$, and $N = 5$ and $N = 
6$.  This behavior is not observed in analogous experiments carried out for 
the scalar Helmholtz problem in random media (cf. \cite{Feng_Lin_Lorton_15}).  
This behavior might lead one to believe that only odd mode functions contribute 
to the solution, but we think that this simple explanation may not be
correct since two previous mode functions are used to produce a new mode 
function at every step (cf. \eqref{eq3.4} and \eqref{Eq:Discrete_Mode}).  On 
the other hand, this is an interesting phenomenon and will be investigated more 
fully in the near future.
\begin{figure}[htb]
\centering
\includegraphics[width=2.5in,height=1.8in]{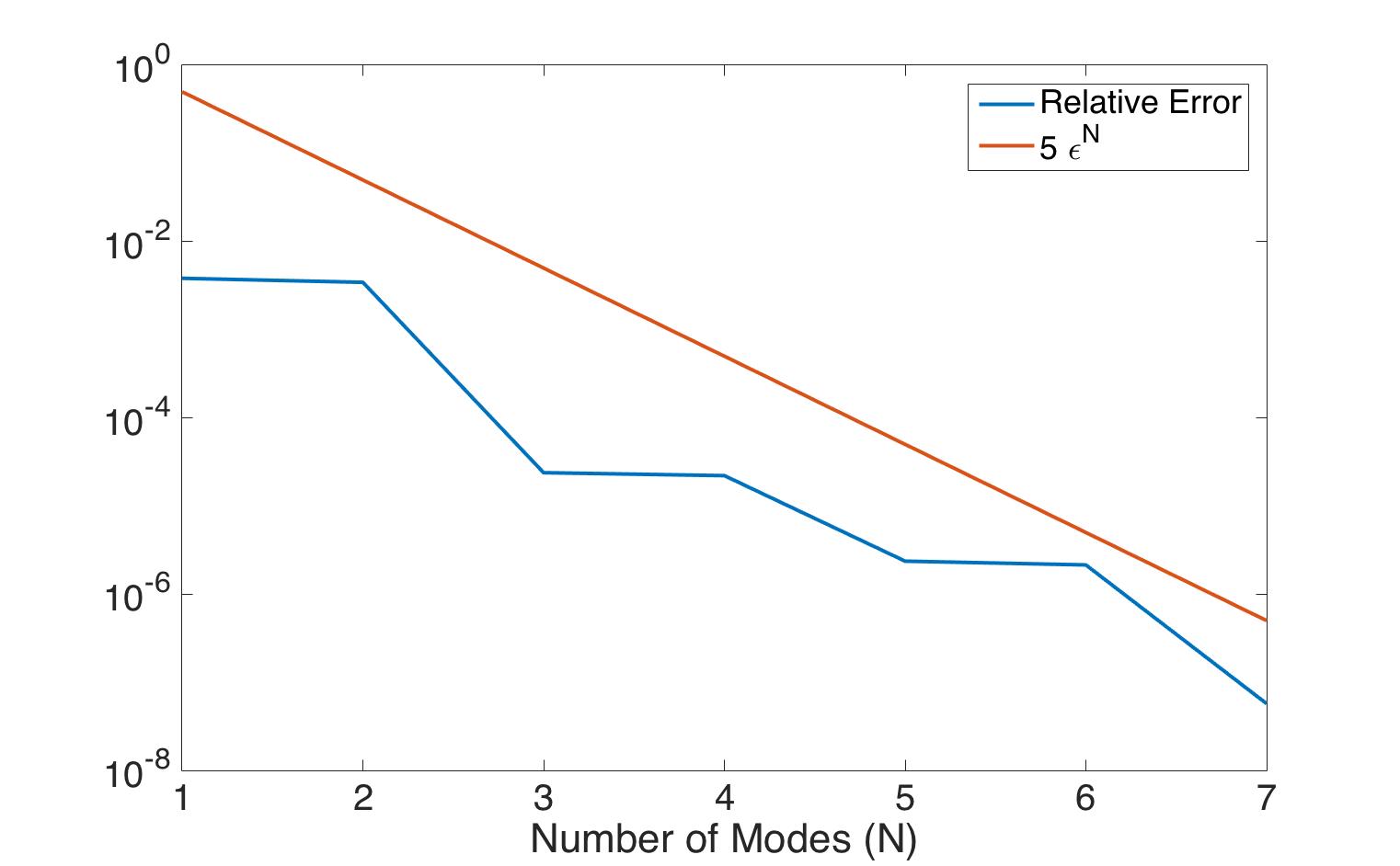}
\caption{$L^2$-norm error between $\Psi^h_N$ computed using MCIP-DG 
with the multi-modes expansion and $\tilde{\Psi}^h$ computed using the 
classical MCIP-DG. The vertical access is given in a log scale.}
\label{fig:ModesVsFullMonte}
\end{figure}

Table \ref{table:ModesVsFullMonte} summarizes the computation time used to 
obtain $\bfPsi^h_N$ using Algorithm 2 with various values of $N$ and the 
computation time used to obtain $\tilde{\bfPsi}^h$ using Algorithm 1.  All 
experiments are performed on the same iMac computer with a 2 GHz Intel Core i7 
processor.  From Table \ref{table:ModesVsFullMonte} we observe that Algorithm 2 
produces accurate approximations with far less computation time than Algorithm 
1.  In fact, Algorithm 2 improves performance by at least one order of 
magnitude.  We also observe that the computation time used for Algorithm 2 
increases linearly as the number of modes $N$ is increased.  This is to be expected.

\begin{table}[htb]
\centering
\begin{tabular}{| c | c |}
	\hline
	Approximation & CPU Time (s) \\
	\hline
	$\tilde{\Psi}^h$ & $21680$ \\
	\hline
	$\Psi^h_1$ & $149$ \\
	\hline
	$\Psi^h_2$ & $335$ \\
	\hline
	$\Psi^h_3$ & $522$ \\
	\hline
	$\Psi^h_4$ & $709$ \\
	\hline
	$\Psi^h_5$ & $897$ \\
	\hline
	$\Psi^h_6$ & $1085$ \\
	\hline
	$\Psi^h_7$ & $1272$ \\
	\hline
\end{tabular}
\caption{CPU times required to compute the MCIP-DG multi-modes approximation $
\Psi^h_N$ 
and classical MCIP-DG approximation $\tilde{\Psi}^h$.} 
\label{table:ModesVsFullMonte}
\end{table}

For our second set of numerical experiments our goal is to check the 
accuracy of Algorithm 2 when the size of the perturbation parameter $\veps$ is 
allowed to grow larger.  In this set of experiments the frequency $k$ is set 
to be $10$ and the perturbation parameter is tested at $\veps = 0.05, 0.1, 0.5, 0.8$.  

Table \ref{table:ModesVSFullMonteError3} shows the relative error associated 
$\bfPsi^h_3$ produced by Algorithm 2 using $N = 3$ modes.  In this table we 
observe that the approximations associated with $\veps = 0.05$ and $0.1$ are
accurate as demonstrated by their small relative errors. For $\veps = 0.5$ and $0.8$ 
Theorem \ref{thm:MM_Monte_Carlo_Error} implies a larger number of modes $N$ might be
needed to produce more accurate approximations.  With this point in mind, Table 
\ref{table:ModesVSFullMonteError4} records the relative error associated  with
these values of $\veps$ along with a larger number of modes $N$.  From this table we 
observe that for $\veps = 0.5$ the relative error does not strictly decrease, 
but when $N = 7$ it does produce a more accurate approximation.  For $\veps = 0.8$ 
using a larger $N$ does not seem to help decrease the relative error.  This is to 
be expected since our convergence theory requires $\veps$ to be small relative to the 
size of $k$ and thus for $\veps$ large, Algorithm 2 will no longer produce accurate solutions.

Lastly, Figures \ref{fig:Fig2a}--\ref{fig:Fig2d} show the solutions $\bfPsi^h_7$ 
produced by Algorithm 2 and sample realizations $\bfU^h_7$ for $k = 10$, $h = 1/20$, 
and $\veps = 0.05$.
\begin{table}[htbp]
\centering
\begin{tabular}{| c || c | c | c | c |}
        \hline
        $\veps$ & $0.05$ & $0.1$ & $0.5$ & $0.8$ \\
        \hline
        Relative $L^2$ Error & $1.4559 \times 10^{-5}$ & $9.7719 \times 10^{-5}$ &
        $0.0399$ & $0.2216$ \\
        \hline
\end{tabular}
\caption{$L^2$-norm relative error between the multimodes expansion 
approximation $\Psi^h_3$ 
and the classical Monte Carlo approximation $\tilde{\Psi}^h$. 
}\label{table:ModesVSFullMonteError3}
\end{table}

\begin{table}[htbp]
\centering
\begin{tabular}{| c || c | c | c | c |}
         \hline
	$\veps$ & $N = 4$ & $N = 5$ & $N = 6$ & $N = 7$ \\
	\hline
	\hline
	$0.5$ & $0.1037$ & $0.0154$ & $0.0450$ & $0.0069$\\
	\hline
	$0.8$ & $0.5218$ & $0.2172$ & $0.5644$ & $0.2155$  \\
	\hline
\end{tabular}
\caption{$L^2$-norm relative error between the multimodes expansion 
approximation $\Psi^h_N$ 
and the classical Monte Carlo approximation $\tilde{\Psi}^h$.
}\label{table:ModesVSFullMonteError4}
\end{table}

\begin{figure}[htb]
\centerline{\includegraphics[width=2.2in,height=1.8in]{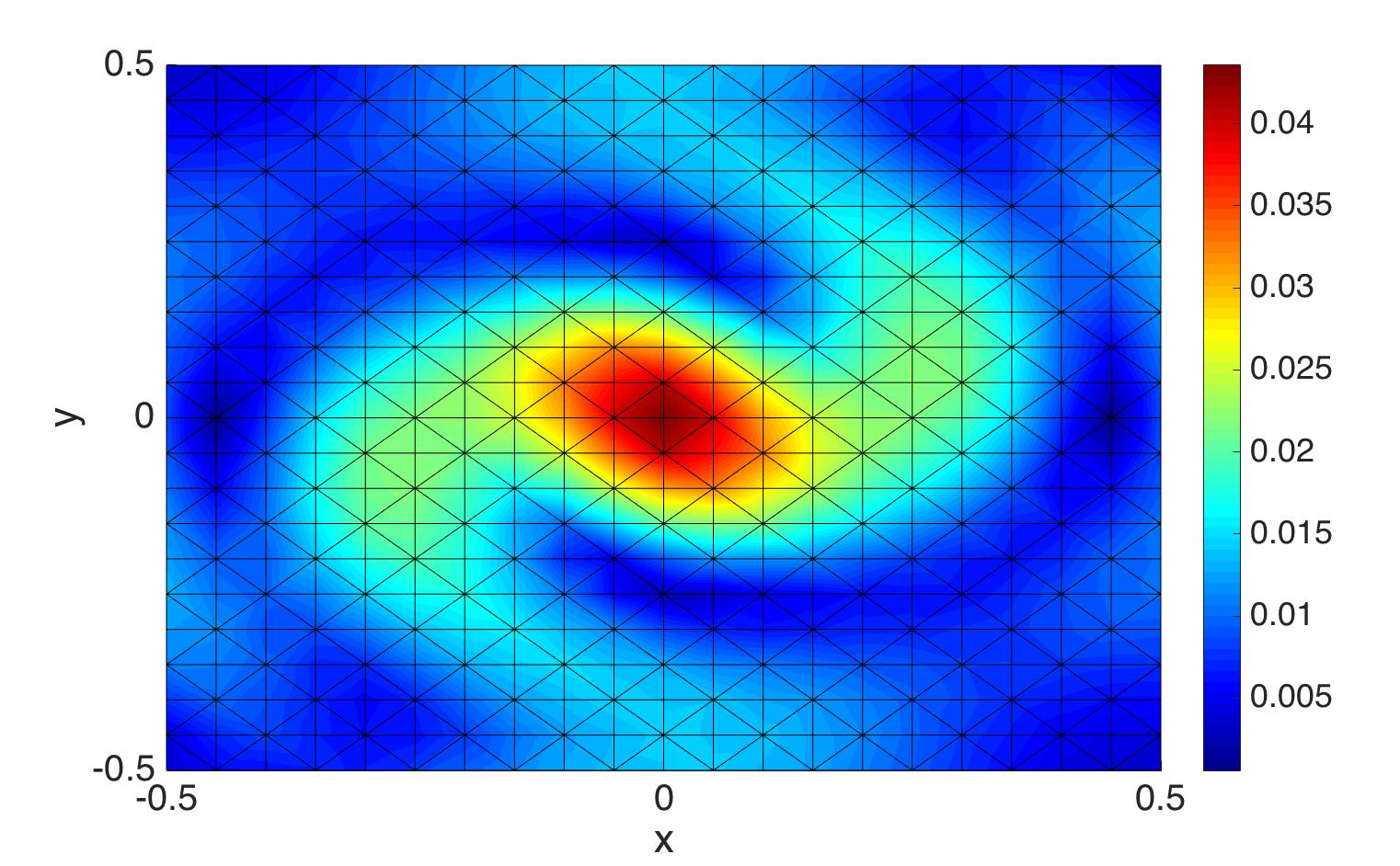} 
\includegraphics[width=2.2in,height=1.8in]{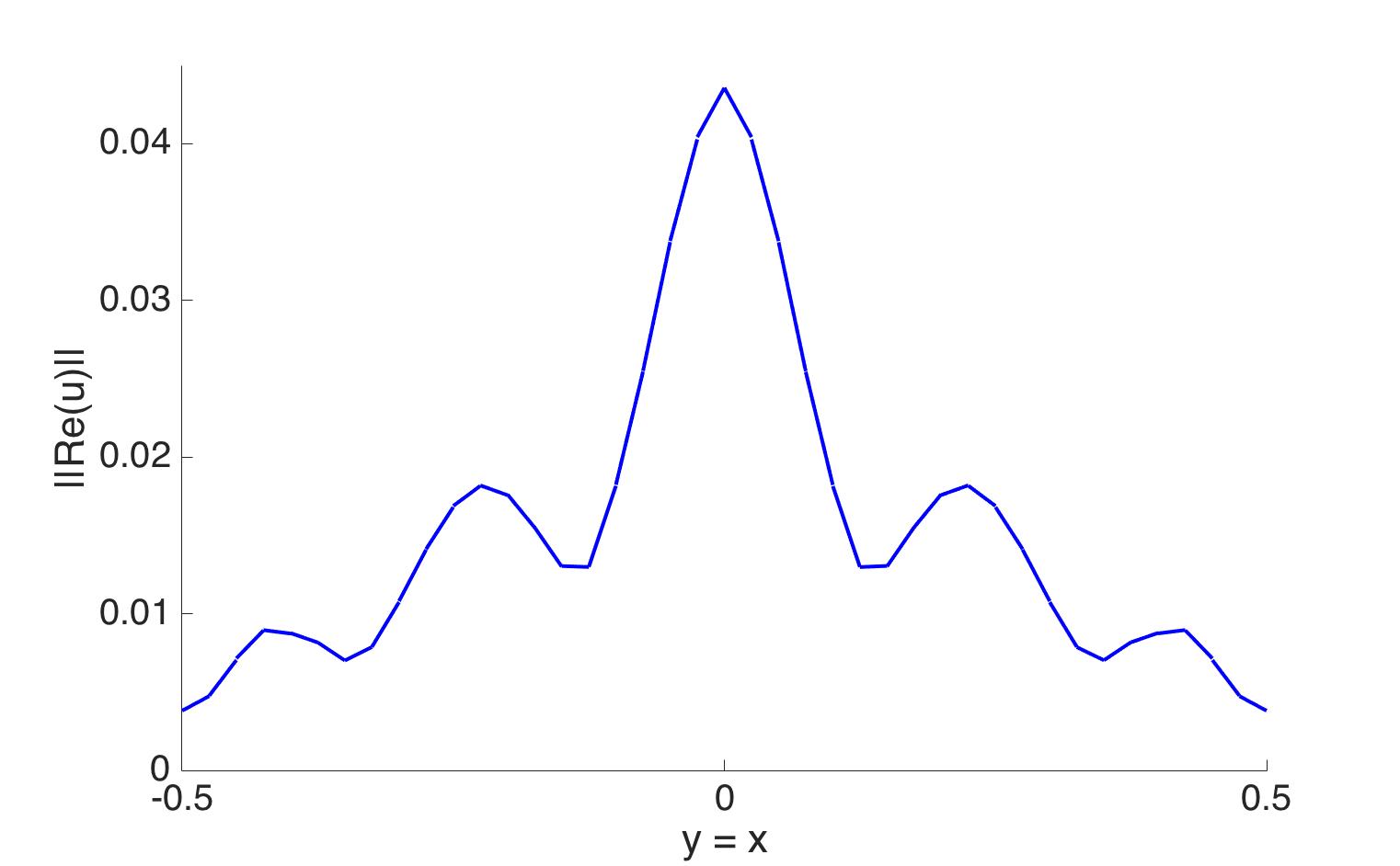}}
\caption{Plot of the statistical average $\Re \big(\Psi^h_7 \big)$ on the 
domain $[-0.5,0.5] \times [-0.5,0.5]$ (left) and on the cross section $y = x$ 
(right) for $k = 10$, $h = 1/20$, $\veps = 0.05$, and $M=1000$.}  
\label{fig:Fig2a}
\end{figure}

\begin{figure}[htb]
\centerline{\includegraphics[width=2.2in,height=1.8in]{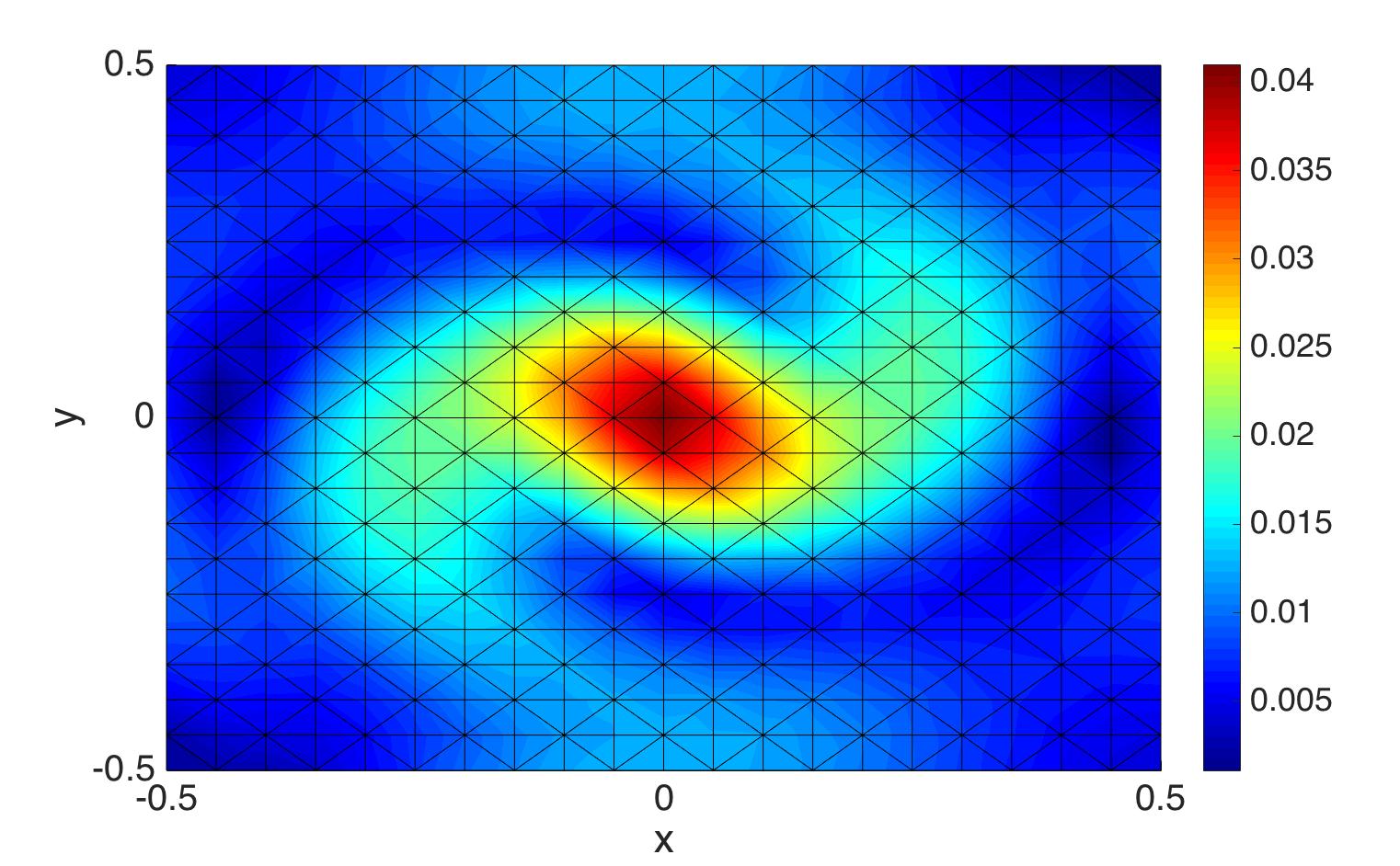} 
\includegraphics[width=2.2in,height=1.8in]{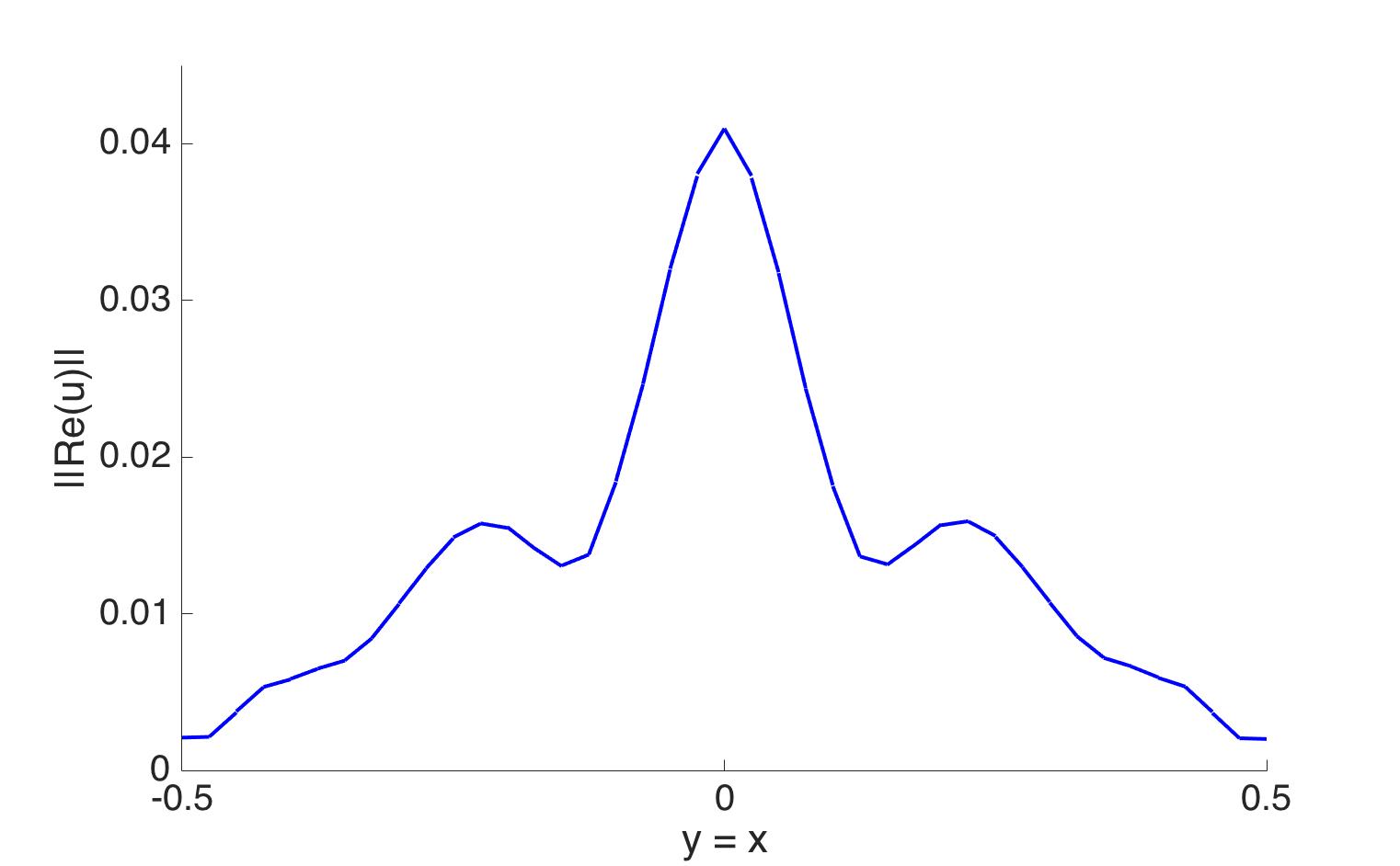}}
\caption{Plot of the sample realization $\Re \big(\bfU^h_7 \big)$ on the domain 
$[-0.5,0.5] \times [-0.5,0.5]$ (left) and on the cross section $y = x$ (right) 
for $k = 10$, $h = 1/20$, $\veps = 0.05$, and $M=1000$.}  
\label{fig:Fig2b}
\end{figure}

\begin{figure}[htb]
\centerline{\includegraphics[width=2.2in,height=1.8in]{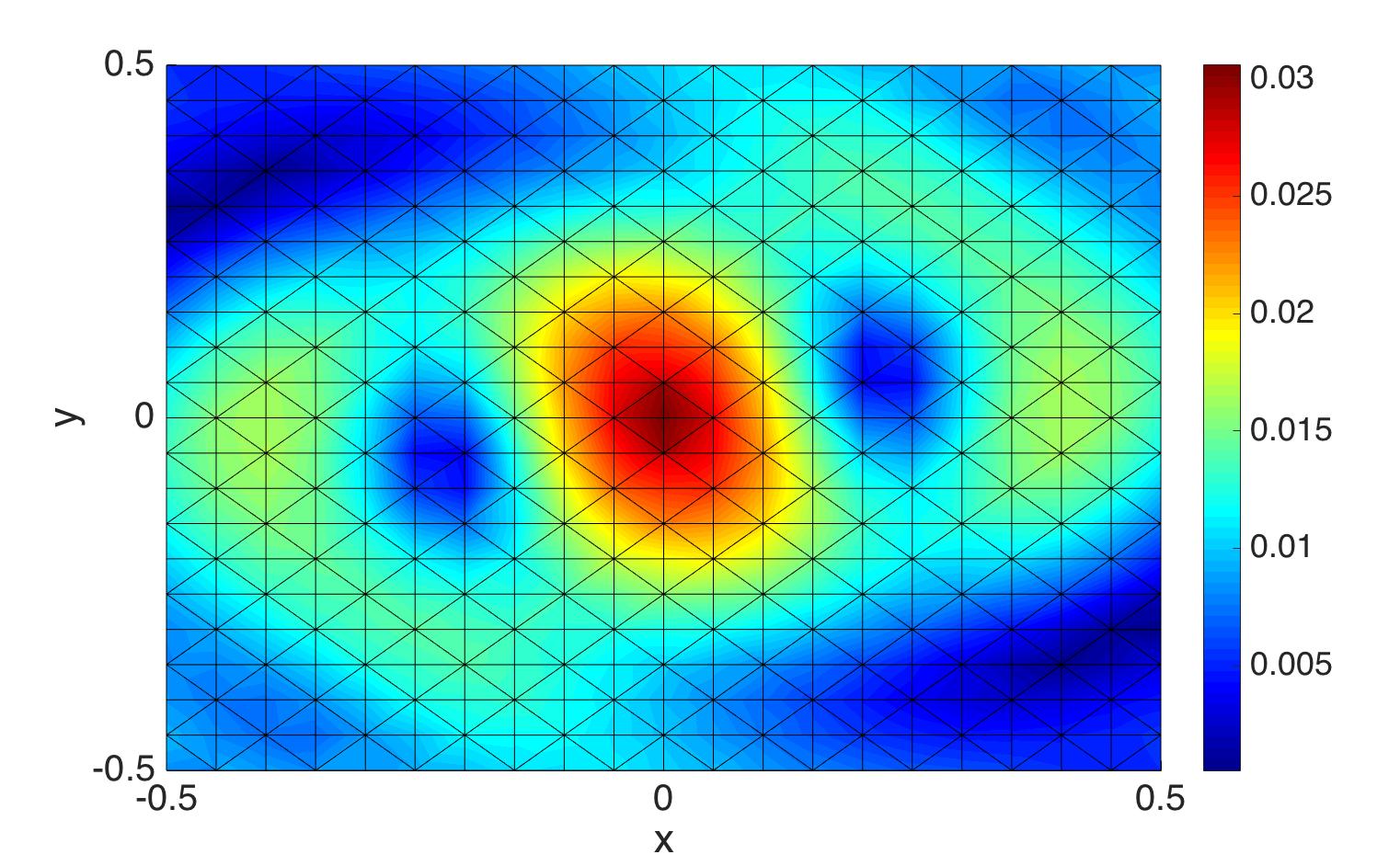} 
\includegraphics[width=2.2in,height=1.8in]{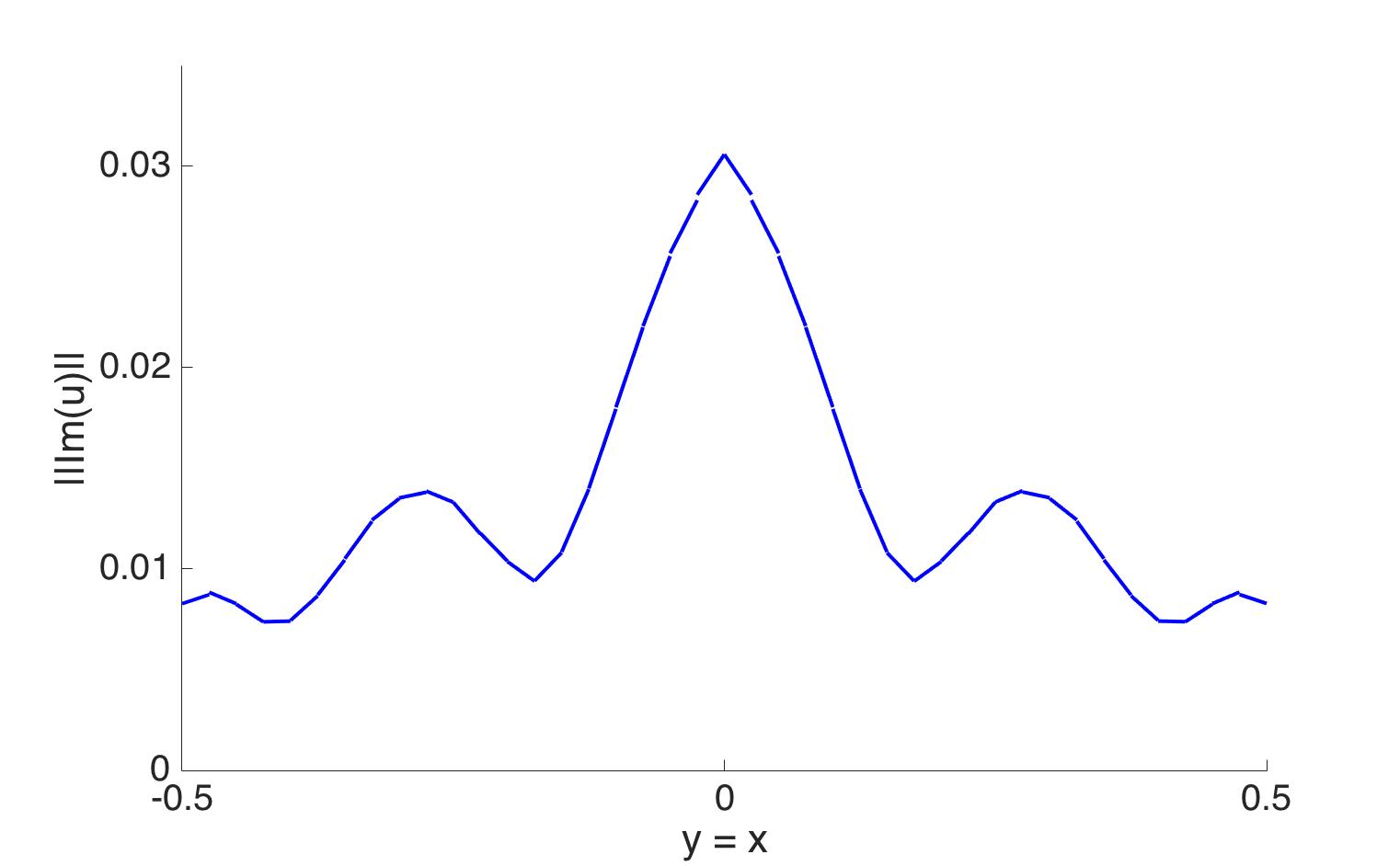}}
\caption{Plot of the statistical average $\Im \big(\Psi^h_7 \big)$ on the 
domain $[-0.5,0.5] \times [-0.5,0.5]$ (left) and on the cross section $y = x$ 
(right) for $k = 10$, $h = 1/20$, $\veps = 0.05$, and $M=1000$.}  
\label{fig:Fig2c}
\end{figure}

\begin{figure}[htb]
\centerline{\includegraphics[width=2.2in,height=1.8in]{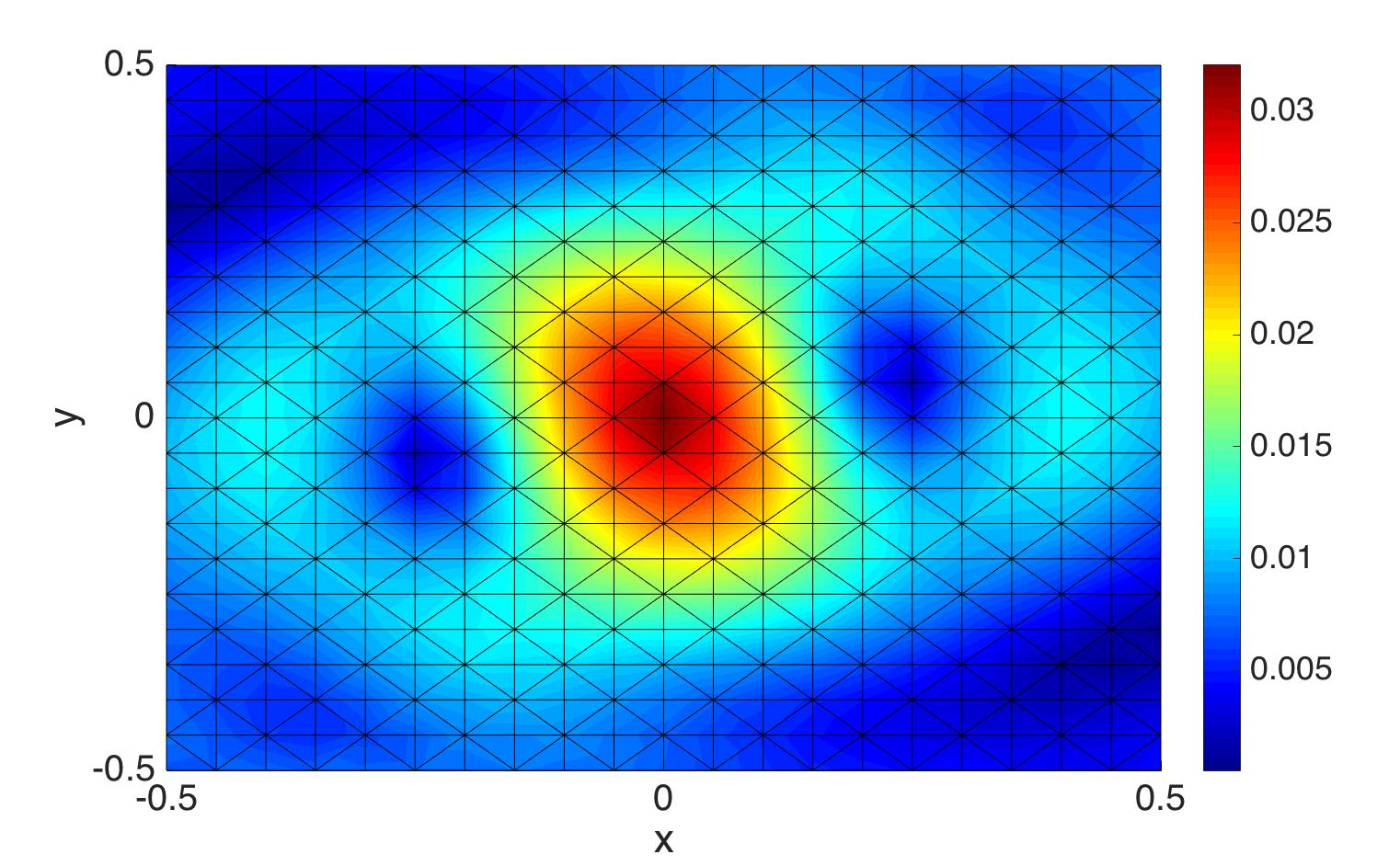} 
\includegraphics[width=2.2in,height=1.8in]{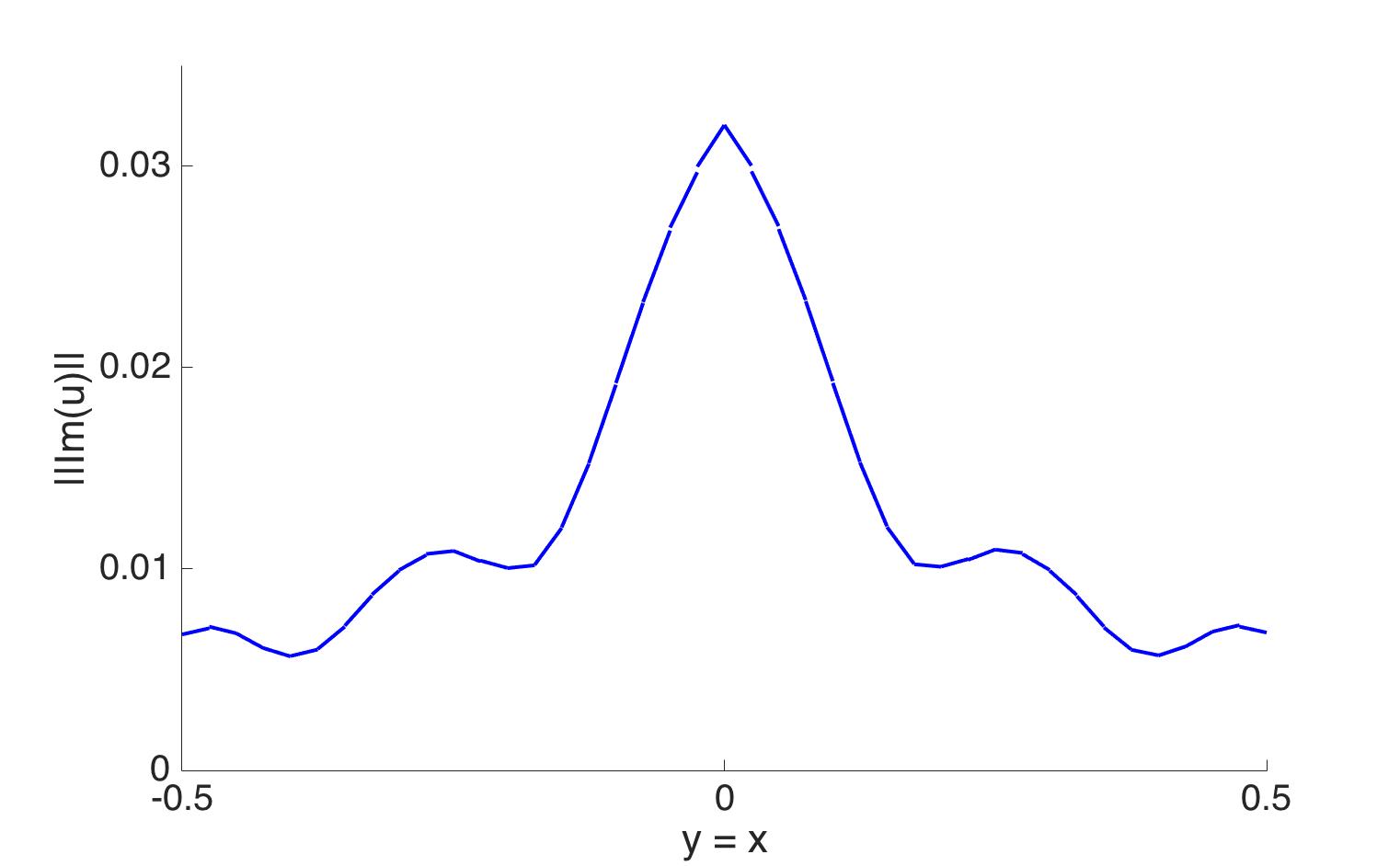}}
\caption{Plot of the sample realization $\Im \big(\bfU^h_7 \big)$ on the domain 
$[-0.5,0.5] \times [-0.5,0.5]$ (left) and on the cross section $y = x$ (right) 
for $k = 10$, $h = 1/20$, $\veps = 0.05$, and $M=1000$.}  
\label{fig:Fig2d}
\end{figure}

%
%
%
%

%

\section{Extension to more general random media}\label{sec-7}

The multi-modes Monte Carlo IP-DG method we developed above is applicable only to
weakly random media in the sense that the coefficient $\alpha$ in the SPDE system must have  
the form $\alpha(\omega,x)=\alpha_0(x) +\veps\eta(\omega,x)$ and $\veps$ is not large.
For more general random media, its density $\rho$ or the coefficient
$\alpha=\sqrt{\rho}$ may not have the required ``weak form". A natural question 
is whether and how the above multi-modes Monte Carlo IP-DG  method can be extended 
to cover more general and non-weak random media. A short answer to this question is positive.
To this end, our  main idea for overcoming this difficulty is first to rewrite 
$\alpha(x,\omega)$ as the required form $\alpha_0(x)+\veps \eta(\omega,x)$, 
then to apply the above weakly random media framework.
There are at least two approaches to do such a re-writing, the first one is to utilize the
well-known Karhunen-Lo\`eve expansion and the second is to use
a recently developed stochastic homogenization theory \cite{DGO}. Since the second approach is
more involved and lengthy to describe, below we only outline the first approach.

For many geoscience and material science applications, the random media can be 
described by a Gaussian random field \cite{Fouque_Garnier_Papanicolaou_Solna_07, Ishimaru_97, Lord_Powell_Shardlow}.
Let $\overline{\alpha}(x)$ and $C(x,y)$ denote the mean and covariance
function of the Gaussian random field $\alpha(\ome,x)$, respectively. Two of the most widely 
used covariance functions in geoscience and materials
science are $C(x,y)=\exp(|x-y|^m/\ell)$ for $m=1,2$ and $0<\ell<1$
(cf. \cite[Chapter 7]{Lord_Powell_Shardlow}. Here $\ell$ is often called correlation length which
determines the range of the noise.  The well-known Karhunen-Lo\`eve expansion
for $\alpha(\ome,x)$ takes the following form (cf. \cite{Lord_Powell_Shardlow}):
\[
\alpha(\omega,x )= \overline{\alpha}(x) + \sum_{k=1}^\infty \sqrt{\lambda_k} \phi_k(x) \xi_k(\omega),
\]
where $\{(\lambda_k, \phi_k)\}_{k\geq 1}$ is the eigenset of the (self-adjoint) covariance 
operator and $\{\xi_k  \sim  N(0,1) \}_{k\geq 1}$ are i.i.d. random variables. It can be 
shown that in many cases there holds $\lambda_k=O(\ell^r)$ for some $r>1$ depending on 
the spatial domain $D$ where the PDE is defined (cf. \cite[Chapter 7]{Lord_Powell_Shardlow}),
that is the case when $D$ is rectangular.  Consequently, we can write
\[
\alpha(\omega,x)=\overline{\alpha}(x)+ \sqrt{\lambda_1} \zeta(x,\omega), \qquad 
\zeta(x,\omega):= \sum_{k=1}^\infty\sqrt{ \frac{\lambda_k}{\lambda_1} }\, \phi_k(x) \xi_k(\omega),
\]
Thus, setting $\varepsilon=O(\ell^{\frac{r}{2}})$ gives rise to 
$\alpha(\ome,x)=\overline{\alpha} + \varepsilon \zeta$,
which is the required ``weak form" consisting of a deterministic background
field plus a small random perturbation.  So we just showed that in many cases a 
given random field $\alpha$ can be rewritten into the required ``weak form".
Therefore, our multi-modes Monte Carlo IP-DG method can still be applied to
such general random media.

It should be pointed out that the classical Karhunen-Lo\`eve expansion may be replaced by other
types of expansion formulas which may result in more efficient multi-modes Monte Carlo methods.
Finally, we also remark that the IP-DG method can be replaced by
any other space discretization method such as finite difference, finite element, and
spectral method in Algorithm 2.


\bibliographystyle{plain}
\bibliography{bibliography}

\end{document}